\documentclass[a4paper, 10pt]{amsart}

\usepackage{latexsym, amssymb,amsmath, amsthm,amsfonts,color,tikz,lscape,geometry}
\usetikzlibrary{math}
\usepackage{pdflscape}
\usepackage{rotating}

\newcommand{\kk}{\Bbbk}
\newcommand{\kv}{{\kk[\mathcal{V}]}}

\newcommand{\kvg}{{\kk[\mathcal{V}]^{G}}}

\newcommand{\spec}{\mathrm{Spec}}
\newcommand{\vv}{\mathcal{V}}

\newcommand{\M}{{\mathcal{M}}}

\newcommand{\bA}{{\mathbf{A}}}
\newcommand{\bB}{{\mathbf{B}}}
\newcommand{\bC}{{\mathbf{C}}}

\newcommand{\bX}{{\mathbf{X}}}

\newcommand{\ba}{{\mathbf{a}}}
\newcommand{\bb}{{\mathbf{b}}}

\newcommand{\cc}{\mathcal{C}}

\def\SL{\operatorname{SL}}

\def\GL{\operatorname{GL}}
\def\Ga{{\mathbb G}_{a}}

\def\Hom{\operatorname{Hom}}
\def\id{\operatorname{id}}

\def\chr{\operatorname{char}}

\def\Z{\mathbb{Z}}

\def\C{\mathbb{C}}

\def\Rep{\operatorname{Rep}}
\def\im{\operatorname{im}}

\def\Ext{\operatorname{Ext}}

\def\Tr{\operatorname{Tr}}
\def\ara{\operatorname{ara}}
\def\rk{\operatorname{rk}}
\def\cdim{\operatorname{cdim}}
\def\sdim{\operatorname{sdim}}

\newtheorem{Lemma}{Lemma}[section]
\newtheorem{Theorem}[Lemma]{Theorem}

\newtheorem{cor}[Lemma]{Corollary}
\newtheorem{prop}[Lemma]{Proposition}
\newtheorem{conj}[Lemma]{Conjecture}

\theoremstyle{definition}
  \newtheorem{Def}[Lemma]{Definition}  

\theoremstyle{remark}
  \newtheorem{rem}[Lemma]{Remark}
\newtheorem{eg}[Lemma]{Example}

\newtheoremstyle{Acknowledgments}
  {}
    {}
     {}
     {}
    {\bfseries}
    {}
     {.5em}
     {\thmname{#1}\thmnumber{ }\thmnote{ (#3)}}
\theoremstyle{Acknowledgments}
\newtheorem{ack}{Acknowledgments.}

\title{The separating variety for matrix invariants}

\author{Jonathan Elmer}
\address{Middlesex University\\
The Burroughs, Hendon, London\\
NW4 4BT UK}
\email{j.elmer@mdx.ac.uk}

\date{\today}

\begin{document}

\maketitle

\begin{abstract} 
Let $G$ be a linear algebraic group defined over an algebraically closed field $\kk$, and let $\vv$ be a vector space on which $G$ acts linearly. Introduced in \cite{KemperCompRed}, the separating variety $\mathcal{S}_{G,\vv}$ is the subvariety of $\vv^2$ consisting of pairs of points indistinguishable by invariant polynomials in $\kvg$. Its geometry places restrictions on the existence of small separating sets, i.e. sets of invariants which distinguish the same points as the full algebra of invariants.

 It is well known that the cardinality of a separating set for $\kvg$ is bounded below by $\dim(\kvg)$. A separating set of cardinality $\dim(\kvg)$ is called a polynomial separating set, and a separating set of cardinality $\dim(\kvg)+1$ is called a hypersurface separating set. 

The purpose of this article is to study the separating variety in the important special case where $G=\GL_p(\C)$ acts on the set $\vv = \M_p^n$ of $n$-tuples of $p \times p$ matrices by simultaneous conjugation. We define a purely combinatorial poset, $\mathcal{P}_{p,n}$, whose maximal elements are in 1-1 correspondence with the irreducible components of $\mathcal{S}_{G,\vv}$. We show that $\mathcal{S}_{G,\vv}$ is a variety of dimension $(n+1)p^2-1$, and determine its subdimension (i.e. the dimension of its smallest component) for all $n$ and $p$. In particular we show the subdimension is $(n+1)p^2-p$ if $n \geq 3$, or $n \geq 2$ and $p \geq 4$. In the case $n \geq 3$, we give a formula for the number of components of given codimension in $\mathcal{S}_{G,\vv}$.
We give explicit decompositions of  $\mathcal{S}_{G,\vv}$ for all $n$ where $p=2,3$ or $4$.

Our results in particular show that when $n \geq 2$ and $p \geq 4$, or $n \geq 3$ and $p = 3$, $\C[\vv]^G$ does not contain a polynomial or hypersurface separating set. It was proven in \cite{ElmerMatrixSepVar} that the same is true if $n \geq 4$ and $p=2$. 

The author made a conjecture in \cite{ElmerSemiSepVar} generalising the Skronowski-Weyman theorem for representations of quivers. The results of this paper prove that conjecture in two important special cases: for the quiver with one vertex and an arbitrary number, $n$, of loops, and for the quiver with two vertices and $n$ arrows between them.
\end{abstract}

\section{Background}
\subsection{Representations of quivers}
A {\it quiver} is a quadruple $Q = (Q_v,Q_a,t,h)$, consisting of two ordered sets $Q_v$ (vertices)  and $Q_a$ (arrows), along with two functions $t, h: Q_a \rightarrow Q_v$ (tail and head respectively). It is usually visualised as a directed graph with a node for each element of $Q_v$, and for each $a \in Q_a$ a directed edge leading from $t(a)$ to $h(a)$. We will deal only with finite quivers, so we may choose an order on the vertices and on the arrows, setting $Q_v = \{x_1, \ldots, x_k\}$ and $Q_a = \{a_1, \ldots, a_n\}$.

Let $\kk$ be any field.  A {\it representation} $V$ of the quiver $Q$ over $\kk$ is an assignment to each vertex $x \in Q_v$ of a vector space $V(x)$, and to each arrow $a \in Q_a$ of a linear map $V(a): V(t(a)) \rightarrow V(h(a))$.  We write $\alpha = (\alpha(x): x \in Q_v)$ where $\alpha(x) = \dim(V(x))$ for all $x \in Q_v$; this is called the {\it dimension vector} of $V$. A homomorphism $\phi: V \rightarrow W$ between representations $V, W$ of $Q$ is a collection of linear maps $(\phi(x): x \in Q_v)$ such that $\phi(x): V(x) \rightarrow W(x)$
for all $x \in Q_v$, and for all $a \in Q_a$ the following diagram commutes:\\

\begin{center}
\begin{tikzpicture}
\node[left] at (0,0){$V(t(a))$};
\node[right] at (3,0){$V(h(a))$};
\node[left] at (0,-2){$W(t(a))$};
\node[right] at (3,-2){$W(h(a))$};
\draw[->] (0,0)--(3,0);
\node[above] at (1.5,0){$V(a)$};
\draw[->] (0,-2)--(3,-2);
\node[above] at (1.5,-2){$W(a)$};
\draw[->] (-0.5,-0.5) -- (-0.5,-1.5);
\node[left] at (-0.5,-1){$\phi(t(a))$};
\draw[->] (3.5,-0.5) -- (3.5,-1.5);
\node[right] at (3.5,-1){$\phi(h(a))$};
\end{tikzpicture}
\end{center}

We write $\Hom_Q(V,W)$ for the set of homomorphisms from $V$ to $W$, which is naturally a $\kk$-vector space. It is clear that the composition of two homomorphisms is again a homomorphism, and there is an obvious identity homomorphism $\id: V \rightarrow V$. Thus, we have a category $\Rep_{\kk}(Q)$ of representations of $Q$ over $\kk$. Denote by $\Rep_{\kk}(Q,\alpha)$ the set of representations of $Q$ with dimension vector $\alpha$. To ease notation we will drop the $\kk$ where the field is known. 

A homomorphism $\phi: V \rightarrow W$ is an isomorphism if $\phi(x)$ is an isomorphism for each $x \in Q_v$.  We write $V \cong W$ if $V$ and $W$ are isomorphic. We say that $W$ is a subrepresentation of $V$ if $W(x) \leq V(x)$ for all $x \in Q_v$ and the collection of inclusion maps $i(x): W(x) \hookrightarrow V(x)$ is a homomorphism $W \rightarrow V$. We write $W \leq V$ if $W$ is a subrepresentation of $V$. It is clear that $V \leq V$ and $0 \leq V$, where $0$ denotes the representation of $Q$ in which $0(x)$ is the zero space for all $x \in Q_v$ and $0(a)$ is the zero map for all $a \in Q_a$. We say $V$ is {\it simple} if it has no subrepresentations other than these. Further, if $W \leq V$ then we may define a quotient representation $V/W$ as follows: $(V/W)(x) = V(x)/W(x)$ for all $x \in Q_v$ and $(V/W)(a)$ is the map $V(t(a))/W(t(a)) \rightarrow (V(h(a))/W(h(a))$ induced by $V(a)$ (the commuting of the diagram above ensures this is well-defined).

There is a natural notion of direct sum for representations of a given quiver: if $V$ and $W$ are representations of $Q$ over $\kk$ with dimension vectors $\alpha$ and $\beta$ respectively, then $V \oplus W$ is the representation of $Q$ over $\kk$ with dimension vector $\alpha+\beta \in \Z^k$ defined by
\[(V \oplus W)(x) = V(x) \oplus W(x)\] for all $x \in Q_v$ and
\[(V \oplus W)(a) = V(a) \oplus W(a)\] for all $a \in Q_a$. One can show that this notion of direct sum makes $\Rep_{\kk}(Q)$ into a $\kk$-linear abelian category.  If $V$ is isomorphic to a direct sum of simple representations we say $V$ is {\it semisimple}.

A representation is said to be {\it indecomposable} if it is not isomorphic to the direct sum of two non-trivial representations. The Krull-Remak-Schmidt Theorem \cite[Theorem~1.7.4]{DerksenWeyman} states that every representation of $Q$ may be written as a direct sum of indecomposable representations and this decomposition is unique up to reordering the summands. Thus, describing the representations of $Q$ up to isomorphism is reduced to the problem of describing the isomorphism classes of indecomposable representations. The quiver $Q$ is said to have:

\begin{enumerate}
\item {\it finite representation type} if $Q$ has only finitely many isomorphism classes of indecomposable representations;
\item {\it tame representation type} if the indecomposable representations of $Q$ in each dimension vector up to isomorphism occur in finitely many one-parameter families;
\item {\it wild representation type} otherwise.
\end{enumerate}

Let $V \in \Rep(Q, \alpha)$. If $V$ is not simple then we may find a subrepresentation $V' \leq V$ such that $V/V'$ is simple. Iterating, we may find a composition series
\[0 = V_0 \leq V_1 \leq V_2 \leq \cdots \leq V_r = V\] in which $S_i:=V_i/V_{i-1}$ is simple for all $i \geq 1$. By the Jordan-H\"older theorem the simple factors $S_i$ are unique up to order. We define $V_{ss}:= \bigoplus_{i=1}^r S_i$, noting that $V$ is semisimple if and only if $V \cong V_{ss}$. 

\subsection{Matrix Interpretation}

For any $p,q \geq$ we let $\M_{p,q}(\kk)$ denote the set of $p \times q$ matrices over $\kk$. We will usually write simply $\M_{p,q}$ for ease of notation, and $\M_{p}$ for $\M_{p,p}$. Let $V$ be a representation of the quiver $Q$ with $|Q_v| = k$, $|Q_a|=n$. By choosing a basis of each vector space $V(x)$, we may identify $V$ with the $n$-tuple of matrices $$\bA = (A_1, A_2, \ldots, A_n)$$ where $A_j \in \M_{\alpha(h(a_j)),\alpha(t(a_j))}$ for all $j$ is the matrix representing $V(a_j)$ with respect to the chosen basis. Choosing a different basis is tantamount to replacing $\bA$ with 
$$g \cdot \bA:= (g_{t(a_1)} A_1 g^{-1}_{h(a_1)},  g_{t(a_2)} A_2 g^{-1}_{h(a_2)}, \ldots,g_{t(a_n)} A_n g^{-1}_{h(a_n)} )$$ where
$$g = (g_{x_1},g_{x_2}, \ldots, g_{x_k}) \in  \GL_{\alpha}(\kk):= \prod_{i=1}^k \GL_{\alpha(x_i)}(\kk)$$ is the $k$-tuple of change of basis matrices where $g_{x_i}$ describes the change of basis on $V(x_i)$. Thus, we have an action of $\GL_{\alpha}(\kk)$ on 
\[\vv_{\alpha}:= \Pi_{i=1}^n \M_{\alpha(h(a_i)),\alpha(t(a_i))}\] and a pair of $n$-tuples of matrices $\bA, \bA' \in \vv_{\alpha}$ represent isomorphic representations of $Q$ if and only if they lie in the same $\GL_{\alpha}(\kk)$-orbit.

We will apply the language of representation theory to $\vv_{\alpha}$; thus, $\bA$ will be called {\it simple}, resp. {\it indecomposable}, resp. {\it semisimple} if it represents an isomorphism class of simple, resp. indecomposable, resp. semisimple representations. These notions have the following interpretation:
let $Y_{\alpha}:= \prod_{i=1}^k \kk^{\alpha(i)}$. Then $\vv_{\alpha}$ acts on $Y_{\alpha}$ via
\[\bA (v_1, \ldots, v_k) = (A_1v_{t(a_1)},A_2v_{t(a_2)}, \ldots, A_nv_{t(a_n)}).\]
It is now easy to see that $\bA$ is simple if and only if $Y_{\alpha}$ does not contain a proper nonzero subspace which is fixed by each $A_i$ (an $\bA$-fixed space). Similarly $\bA$ is indecomposable if and only if $Y_{\alpha}$ does not decompose as a direct sum of proper nonzero $\bA$-fixed spaces. 

Now the following describes the orbits of $\GL_{\alpha}$ on $\vv_{\alpha}$ and their closures (see \cite[Section~9.7]{DerksenWeyman}):

\begin{prop}\label{closures}
Let $\bA, \bA' \in \vv_{\alpha}$, representing $V,V' \in \Rep(Q, \alpha)$. Then
\begin{enumerate}
\item $\GL_{\alpha} \cdot \bA$ is closed if and only if $\bA$ is semisimple.

\item $\overline{\GL_{\alpha} \cdot \bA}$ contains a unique closed orbit, namely $\GL_{\alpha} \cdot \bA_{ss}$. Here $\bA_{ss}$ is any matrix representing $V_{ss}$; we usually take it to be block-diagonal, but even this is defined only up to reordering blocks.

\item  $\overline{\GL_{\alpha} \cdot \bA} \cap \overline{\GL_{\alpha} \cdot \bA'}$ is non-empty if and only if  $\GL_{\alpha} \cdot \bA_{ss} =  \GL_{\alpha} \cdot \bA'_{ss}$.
\end{enumerate}
\end{prop}

Let $W \in \Rep(Q,\alpha), W' \in \Rep(Q, \alpha')$. An {\it extension} of $(W, W')$ is a representation $V \in \Rep(Q,\alpha+\alpha')$ such that $W \leq V$ and $V/W \cong W'$. Equivalently, there exists an exact sequence
\[0 \rightarrow W \rightarrow V \rightarrow W' \rightarrow 0.\]
$V$ is said to be a {\it trivial extension} if $V \cong W \oplus W'$.  A pair $V,V'$ of extensions is equivalent if there is a commutative diagram 

\begin{center}
\begin{tikzpicture}
\node[left] at (0,0){$W$};
\node[right] at (3,0){$W'$};
\node at (1.5,0){$V$};
\node[left] at (0,-2){$W$};
\node at (1.5,-2){$V'$};
\node[right] at (3,-2){$W'$};
\node[left] at (-2,0){$0$};
\draw[->] (-2,0)--(-1,0);
\node[left] at (-2,-2){$0$};
\draw[->] (-2,-2)--(-1,-2);
\draw[->] (0,0)--(1,0);
\draw[->] (2,0) -- (3,0);
\draw[->] (4,0)-- (5,0);
\node[right] at (5,0){$0$};
\draw[->] (4,-2)-- (5,-2);
\node[right] at (5,-2){$0$};
\draw[->] (0,-2)-- (1,-2);
\draw[->] (2,-2)--(3,-2);
\node[above] at (1.5,-2){};
\draw[->] (-0.5,-0.5) -- (-0.5,-1.5);
\node[left] at (-0.5,-1){$\id_W$};
\draw[->] (3.5,-0.5) -- (3.5,-1.5);
\node[right] at (3.5,-1){$\id_{W'}$};
\draw[->] (1.5,-0.5)--(1.5,-1.5);
\node[left] at (1.5,-1){$\phi$};
\end{tikzpicture}
\end{center}
where the rows are exact and $\phi \in \Hom_Q(V,V')$ is an isomorphism. Denote the set of equivalence classes of extensions of $(W,W')$ by $\Ext_Q(W,W')$. This can be given a vector space structure: see \cite[Exercise~2.4.2]{DerksenWeyman}.

Suppose $V$ is an extension of $(W,W')$. Then $V(x) = W(x) \oplus W'(x)$ for all $x \in Q_v$. If we identify $V$ with the $n$-tuple $\bA$, then each $A_i$ has a block decomposition
\[A_i = \begin{pmatrix} B_i & C_i \\ 0 & B_i' \end{pmatrix}\] where $B_i \in \M_{\alpha(h(a_i)), \alpha(t(a_i))}$, $B'_i \in \M_{\alpha'(h(a_i)), \alpha'(t(a_i))}$, 
$C_i \in \M_{\alpha(h(a_i)), \alpha'(t(a_i))}$, and the $n$-tuples $\bB$ and $\bB'$ represent $W$ and $W'$ respectively.

Define a mapping $d_{\bB,\bB'}: \prod_{i=1}^k \M_{\alpha'(i),\alpha(i)} \rightarrow \prod_{i=1}^n \M_{\alpha'(h(a_i)), \alpha(t(a_i))}$ as follows:
\[d_{\bB,\bB'}(\bX)_i = B'_i X_{t(i)} - X_{h(i)}B_i .\]
It is clear that $\ker(d_{\bB,\bB'})$ is identified with the set of homomorphisms from $W$ to $W'$. We also have (see \cite[Section~2.4]{DerksenWeyman}):

\begin{prop}\label{ext}
Let \[\bA = \begin{pmatrix} \bB & \bC \\ 0 & \bB' \end{pmatrix}\] and
\[\bA' = \begin{pmatrix} \bB & \bC' \\ 0 & \bB' \end{pmatrix}\]
The following are equivalent:
\begin{enumerate}
\item $\bA$ and $\bA'$ are equivalent as extensions of $\bB$ by $\bB'$;

\item $\bC-\bC' \in \im(d_{\bB,\bB'})$.
\end{enumerate}
In particular, every extension of $\bB$ by $\bB'$ is trivial if and only if $d_{\bB, \bB'}$ is surjective.
\end{prop}

\subsection{Invariant theory and separating invariants}

Now consider a more general situation in which a linear algebraic group $G$ defined over an algebraically closed field $\kk$ acts linearly on an affine $\kk$-variety $\vv$. Let $\kv$ denote the algebra of polynomial functions on $\vv$. Then $G$ acts on $\kv$ according to the formula
\begin{equation}\label{action} (g \cdot f) (v) = f(g^{-1}  v).\end{equation} We denote by $\kvg$ the subalgebra of $\kv$ fixed by this action. For $v, w \in \vv$ and any $f \in \kvg$ we have $f(v) = f(w)$ if $v \in G \cdot w$. If $G$ is not finite, then the orbit $G \cdot v$ is not closed in general in the Zariski topology, and for reasons of continuity we must have $f(v)=f(w)$ for all $v,w \in \vv$ with $\overline{G \cdot v} \cap \overline{G \cdot w} \neq \emptyset$.

If $f \in \kvg$ and $f(v) \neq f(w)$ we say that $f$ {\it separates} $v$ and $w$. We say that $v$ and $w$ are {\it separated by invariants} if there exists an invariant separating $v$ and $w$.  In case $G$ is reductive, we have by a result of Mumford that $f(v) = f(w)$ for all $f \in \kvg$ if and only if $\overline{Gv} \cap \overline{Gw} \neq \emptyset$, see  \cite[Corollary~6.1]{Dolgachev}. In particular, the invariants separate the orbits if $G$ is a finite group.

One can therefore distinguish the orbits of a group action whenever one can find an explicit generating set for $\kvg$, This is a difficult problem in general. For this reason, Derksen and Kemper introduced the following in 2002 \cite[Definition~2.3.8]{DerksenKemper}:

\begin{Def}\label{sepdef}
Let $S \subseteq \kvg$. We say $S$ is a {\it separating set} for $\kvg$ if the following holds for all $v, w \in \vv$:
\begin{equation*} s(v) = s(w) \ \text{for all} \ s \in S \Leftrightarrow f(v) = f(w)  \ \text{for all} \ f \in \kvg. \end{equation*}
\end{Def}
Separating sets of invariants have been an area of much recent interest. In general they have some nicer properties and are easier to construct than generating sets. For example, if $G$ is a finite group acting on a vector space $V$, then the set of invariants of degree $\leq |G|$ form a separating set \cite[Theorem~3.9.14]{DerksenKemper}. This is also true for generating invariants if $\chr(\kk) = 0$ \cite{FleischmannNoetherBound,FogartyNoetherBound} but fails for generating invariants in the modular case. Separating sets for the rings of invariants $\kk[V]^{C_p}$, where $\kk$ is a field of characteristic $p$ and $C_p$ the cyclic group of order $p$ and $V$ is indecomposable were constructed in \cite{SezerCyclic}. Corresponding sets of generating invariants are known only when $\dim(V) \leq 10$ \cite{WehlauCyclicViaClassical}. For the (non-reductive) linear algebraic group $\Ga$ of a field of characteristic zero, separating sets for $\kk[V]^{\Ga}$ for arbitary indecomposable linear representations $V$ were constructed in \cite{ElmerKohls}, andt hese results were extended to decomposable representations in \cite{DufresneElmerSezer}. Even for indecomposable representations, generating sets are known only where $\dim(V) \leq 8$ \cite{Bedratyuk7}.  Finally, for an arbitrary (i.e. non-linear) $\Ga$-variety $\vv$, the algebra of invariants $\kk[\vv]^{\Ga}$ may not be finitely generated, but it is known that there must exist a finite separating set \cite{KemperSeparating} and finite separating sets have been constructed for many examples where $\kk[\vv]^{\Ga}$ is infinitely generated \cite{DufresneKohlsFiniteSep, DufresneKohlsSepVar}. 

Let $S$ be a separating set for $\kvg$.. The size of a separating set is bounded below by the dimension of $\kvg$, see \cite[Section~1.2]{ElmerMatrixSepVar}. A separating set whose size equals the dimension of $\kvg$ is sometimes called a polynomial separating set, because it necessarily generates a polynomial subalgebra of $\kvg$. On the other hand, there always exists a separating set of size $\leq 2\dim(\kvg)+1$, albeit such a separating set may necessarily contain non-homogeneous polynomials; see \cite[Theorem~5.3]{KamkeKemper} for a proof.

The {\it separating variety} was introduced by Kemper in \cite{KemperCompRed}:
\begin{Def}\[ \mathcal{S}_{G,\vv} = \{(v, w) \in \vv^2: f(v)= f(w) \ \text{for all} \ f \in \kvg \} .\]
\end{Def}
In other words, the separating variety is the subvariety of $\vv^2$ consisting of pairs of points which are not separated by any invariant. We note that $\mathcal{S}_{G,\vv}$ is preserved by the natural action of $G^2$ on $\vv^2$. If $G$ is reductive then we have the alternative description

\begin{equation}\label{sepvarred} \mathcal{S}_{G,\vv} = \{(v, w) \in \vv^2: \overline{G \cdot v} \cap \overline{G \cdot w} \neq \emptyset\} .
\end{equation}

Since invariants are constant on orbits. the separating variety always contains the {\it graph}
\[\Gamma_{G,\vv} = \{(v, g\cdot v) \in \vv^2: g \in G, v \in \vv\}\]
of the action of $G$ on $\vv$. This is not closed in general, so we have an inclusion $\overline{\Gamma_{G,\vv}} \subseteq \mathcal{S}_{G,\vv}$. Often this is an equality, but not always.

We define $I_{G,\vv}$ to be the ideal of $\kk[\vv^2]$ consisting of the polynomial functions which vanish on $\mathcal{S}_{G,\vv}$. Clearly this is a radical ideal. Then a separating set can be characterised as a subset $S \subseteq \kk[\vv]^G$ which cuts out the separating variety in $\vv^2$, in other words (see \cite[Theorem~2.1]{DufresneSeparating})

\begin{prop}\label{radversion}
$S \subseteq \kk[\vv]^G$ is a separating set if and only if 
\[V_{\vv^2}(\delta(S)) = \mathcal{S}_{G,\vv}.\]
\end{prop}
where  $\delta: \kk[\vv] \rightarrow \kk[\vv^2] = \kk[\vv] \otimes \kk[\vv]$ is defined by
\[\delta(f) = 1 \otimes f - f \otimes 1.\] 

Equivalently, via the Nullstellensatz, $S$ is a separating set if and only if
\[\sqrt{(\delta(f): f \in S)} = {I_{G,\vv}}.\]

Consequently the size of a separating set for $\vv$ is bounded below by the minimum number of generators of $I_{G,\vv}$ up to radical, that is, the minimum number of elements generating any ideal whose radical is $I_{G,\vv}$. (This is called the {\it arithmetic rank} $\ara(I)$ of an ideal $I$). This idea has been used to determine lower bounds for the size of separating sets in several cases: see for example \cite{DufresneJeffries,DufresneJeffriesTori,ElmerSemiSepVar,ElmerMatrixSepVar,Reimers}.

Recall that the {\it subdimension} $\sdim(V)$ of an algebraic variety $V$ is the dimension of its smallest irreducible component. It follows from Krull's principal ideal theorem that, for a separating set $S \subseteq \kvg$ we have
 \begin{equation}\label{krullbound} |S| \geq 2 \dim(\vv) - \sdim(\mathcal{S}_{G,\vv}). \end{equation}

\subsection{Invariants and semi-invariants of quivers}

We recommend \cite{DerksenWeyman} as a good source for the representation theory of quivers and its connection with invariant theory. We return to the notation of Subsection 1.1. Now let $G$ denote the subgroup $\SL_{\alpha}(\C):= \Pi_{i=1}^k \SL_{\alpha(x_i)}(\C)$ of $\GL_{\alpha}(\C)$. The algebra $\C[\vv_{\alpha}]^G$ is called the {\it algebra of semi-invariants} associated to $Q$ with dimension vector $\alpha$. A remarkable result linking representation type of quivers and invariant theory is the following (see \cite{SkowronskiWeyman} for proof):

\begin{prop}[Skowronski-Weyman]\label{SW}  Let $Q$ be a quiver. Then the following are equivalent:
\begin{enumerate}
\item $Q$ has finite or tame representation type;

\item $\C[\vv_{\alpha}]^G$ is a polynomial ring or hypersurface for each dimension vector $\alpha$.
\end{enumerate}
\end{prop}

Along similar lines, Sato and Kimura \cite{SatoKimura} showed that if $Q$ has finite representation type then $\C[\vv_{\alpha}]^G$ is polynomial for all dimension vectors $\alpha$. Now by \cite[Corollary~4.6]{DufresneKraft}, we have that if an algebraic group $G$ acts on a normal variety $V$, and $\kk$ has characteristic zero, and $\kk[V]^G$ contains a polynomial separating algebra then $\kk[V]^G$ is itself polynomial. So we obtain the following ``separating'' version of Sato and Kimura's result immediately:

\begin{prop}
 Let $Q$ be a quiver. Then the following are equivalent:
\begin{enumerate}
\item $Q$ has finite representation type;

\item $\C[\vv_{\alpha}]^G$ contains a polynomial separating set for each dimension vector $\alpha$.
\end{enumerate}
\end{prop}

 In \cite{ElmerSemiSepVar}, the author made the following conjecture generalising Proposition \ref{SW}:

\begin{conj}\label{conj}  Let $Q$ be a quiver. Then the following are equivalent:
\begin{enumerate}
\item $Q$ has finite or tame representation type;

\item $\C[\vv_{\alpha}]^G$ contains a polynomial or hypersurface separating set, for each dimension vector $\alpha$.
\end{enumerate}
\end{conj}

Here a hypersurface separating set is a separating set generating a hypersurface, i.e. with cardinality $\dim(\C[\vv_{\alpha}]^G)+1$.
Since generating sets are separating sets, the forward direction is known. To prove the conjecture, it remains to show that if $Q$ has wild representation type, then there exists a dimension vector $\alpha$ for which $|S| > \dim(\C[\vv_{\alpha}]^G+1)$ for every separating set $S \subseteq \C[\vv_{\alpha}]^G$.

\begin{eg} Representations of the $n$-Kronecker quiver $Q_n$ with diagram

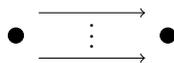
\begin{center}
\begin{tikzpicture}
\draw[fill] (-1,0) circle[radius=0.1];
\draw[fill] (1,0) circle[radius=0.1];
\draw[->] (-0.7,0.3) -- (0.7,0.3);
\node at (0,0.1){$\vdots$};
\draw[->] (-0.7,-0.3) -- (0.7,-0.3);
\end{tikzpicture}
\end{center}

with $n$ arrows and dimension vector $(q,p)$ over $\C$ can be identified with $n$-tuples of $p \times q$ matrices, i.e. elements of $\M_{p,q}^n(\C)$, and two such are isomorphic if they lie in the same $\GL_{p}(\C) \times \GL_q(\C)$ orbit, where the action is by simultaneous left- and right- multiplication. The quiver is of finite type for $n=1$, tame for $n=2$, and wild for $n \geq 3$. Let $G = \SL_p(\C) \times \SL_q(\C)$. \cite[Theorem~1.16]{ElmerSemiSepVar}  shows that the algebra of invariants $\C[\M_{2,2}^n]^G$ does not contain a hypersurface separating algebra if $n\geq 5$, while $\C[\M_{2,2}^n]^G$ is known to be a polynomial algebra for $n \leq 3$, and a hypersurface for $n=4$. To verify the conjecture for this quiver it remains to show that, for $n=3$ and for $n=4$, there exists $p,q$ such that $\C[\M_{p,q}^n]^G$ does not contain a polynomial or hypersurface separating set. 
\end{eg}

\begin{eg}\label{nloop} Representations of the $n$-loop quiver $Q_n$

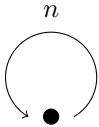
\begin{center}
\begin{tikzpicture}
\draw[fill] (0,0) circle[radius=0.1];

\draw[->] (0.3,0) arc (-60:240:0.6);
\node at (0,1.4) {$n$};
\end{tikzpicture}
\end{center}
where the $n$ represents $n$ separate arrows and $\alpha = p$ can be identified with $n$-tuples of $p \times p$ matrices, i.e. elements of $\M_p^n$, with a pair of such representations isomorphic if and only if those $n$-tuples of matrices are simultaneously conjugate under the action of $\GL_{p}(\C)$. Let $G = \SL_p(\C)$ act on $\M_p(\C)$ by simultaneous conjugation. The semi-invariant rings of such actions were studied in \cite{ElmerMatrixSepVar}. The quiver $Q_n$ is of tame type if $n=1$ and wild otherwise. Theorem 1.4 in loc. cit. shows that $\C[\M_{2}^n]^G$ does not contain a hypersurface separating set for $n \geq 4$, and $\C[\vv_{2}^n]^G$ is known to be polynomial if $n=2$ and a hypersurface if $n=3$. Thus to verify the conjecture for this quiver  it remains to show that, for $n=2$ or $n=3$ there exists some $p$ such that $\C[\M_{p}^n]^G$ does not contain a polynomial or hypersurface separating set.
\end{eg}

\subsection{Matrix invariants and main results}

In this subsection we consider only the $n$-loop quiver $Q_n$ described in Example \ref{nloop}, and its representations over $\C$. In that case we have $\vv = \vv_{\alpha} = \M^n_p(\C)$ for some $p>0$, and we note that the scalar multiples of the identity matrix in $\GL_p(\C)$ act trivially, so that $\C[\vv]^{\SL_p(\C)} = \C[\vv]^{\GL_p(\C)}$. We set $G = \GL_p(\C)$ throughout.

The elements of $\vv$ wil be viewed as $n$-tuples $\bA = (A_1,A_2, \ldots, A_n)$, or as $p \times p$ matrices with entries in $\C^n$. We call these $n$-matrices for short. For $g \in G$ we write the conjugation action as
\[g \cdot \bA:= (gA_1g^{-1}, \ldots, gA_ng^{-1}).\] 

For $1 \leq i,j \leq p$ and $1 \leq k \leq n$, let $a_{ij}^{(k)}$ denote the $i,j$th entry of $A_k$, and write $\ba_{ij} \in \C^n$ for the vector $(a^{(1)}_{ij}, \ldots, a^{(n)}_{ij})$.
If we need to discuss a second element $\bA' \in \vv$ we shall write $\bA' = (A'_1,A'_2, \ldots, A'_n)$ with the $i,j$th element of $A'_k$ being ${a'}^{(k)}_{ij}$ and $\ba'_{ij} = ({a'}^{(1)}_{ij}, \ldots, {a'}^{(n)}_{ij})$.

Now for $1 \leq i,j \leq p$ and $1 \leq k \leq n$, let $x_{ij}^{(k)}$ denote the linear functional $\vv \rightarrow \C$ which picks out the $i,j$th entry of of $A_k$, and introduce generic matrices
\[X_k:= \begin{pmatrix} x_{11}^{(k)} & x_{12}^{(k)} \ldots & x_{1p}^{(k)} \\ x_{21}^{(k)} & x_{22}^{(k)} \ldots & x_{2p}^{(k)} \\
\vdots & \vdots & \vdots \\ x_{p1}^{(k)} & x_{p2}^{(k)} \ldots & x_{pp}^{(k)} 
 \end{pmatrix}.\]
Then we have \[\C[\vv] = \C[x_{ij}^{(k)}: i,j = 1, \ldots, p, k = 1, \ldots, n].\]

Kac \cite[Proposition~4]{KacQuivers} discovered a method of computing the Krull dimension of the algebra of semi-invariants of an arbitrary quiver. Applying his method to the $n$-loop quiver we find:

\begin{prop}\label{dimcvg}
The algebra of invariants $\C[\vv]^G$ has dimension $(n-1)p^2+1$ for $n \geq 2$ and $p>1$. 
\end{prop}

For a finite sequence $I = (i_1,\ldots, i_r)$ it is easy to see that 
\[t_{i_1,i_2,\ldots, i_r}:= \Tr(X_{i_1}X_{i_2}\cdots X_{i_r}) \in \C[\vv]^G.\]
Sibirskii \cite{Sibirskii} showed that these invariants generate $\C[\vv]^G$ as a vector space. LeBruyn and Procesi \cite{LeBruynProcesi} showed that for $p=2$, $\C[\vv]^G$ is minimally generated as an algebra by 
$$\{t_i: 1 \leq i \leq n, t_{ij}: 1 \leq i \leq j \leq n, t_{ijk}: 1 \leq i<j<k \leq n\},$$ a generating set with size $\frac16(n^3+11n)$. In particular, $\C[\vv]^G$ is not polynomial for $n \geq 3$ and not a hypersurface for $n \geq 4$. Kaygorodov et al \cite{Lopatin2x2} showed that this set is a minimal separating set by inclusion, but that does not necessarily mean that it has the smallest possible cardinality among all separating sets. The following was shown in \cite{ElmerMatrixSepVar}:

\begin{prop}
Let $S \subseteq \C[\M^n_2]^G$ be a separating set. Then $|S| \geq 5n-5$.
\end{prop}

This shows that the given set of generators is a separating set of minimum cardinality if $n=3$ but may not be so if $n \geq 4$. It also shows that there is no polynomial or hypersurface separating set for $n \geq 4$, which verifies conjecture \ref{conj} for $Q_n$ when $n \geq 4$.

In the case $p=3$, $n \geq 2$, Abeasis and Pittaluga \cite{AbeasisPittaluga} gave a minimal generating set of $\C[\vv]^G$ with cardinality 
\[3n+5\binom{n}{2}+24\binom{n}{3}+51\binom{n}{4}+47\binom{n}{5}+15 \binom{n}{6}\] where $\binom{n}{k}=0$ if $k>n$.
In particular their result shows that $\C[\vv]^G$ is a hypersurface (but not polynomial) for $n = 2$ and not a hypersurface for $n \geq 3$. See also \cite{Lopatin3x3}. It is not known whether this set is a minimal separating set in general. In the case $p=4$ and $n=2$, Drensky and Sadikova \cite{DrenskySadikova} found a minimal set of bihomogeneous polynomials which generate $\C[\vv]^G$. It is not known whether this set is a minimal generating or separating set.

Our most general results on separating sets are as follows:

\begin{Theorem}\label{mainpxp}
Let $n \geq 3$ and $p \geq 2$, or $n=2$ and $p \geq 4$ and  and let $S \subseteq \C[\M_p^n]^G$ be a separating set. Then $|S| \geq (n-1)p^2+p$.
\end{Theorem}

In order to prove these results, we will determine the subdimension of the separating variety $\mathcal{S}_{G,\M_p^n}$ for all $p \geq 2$ and $n \geq 2$ in the next section. Our main result is as follows:

\begin{Theorem}\label{mainsepvar}
Let $n \geq 3$ and $p \geq 2$, or $n=2$ and $p \geq 4$. Then $\dim(\mathcal{S}_{G,\M_p^n}) = (n+1)p^2-1$ and  $\sdim(\mathcal{S}_{G,\M_p^n}) = (n+1)p^2-p$.
\end{Theorem} 

It is now easy to deduce our main results from the above.

\begin{proof}[Proof of Theorem \ref{mainpxp}]
We have by \eqref{krullbound}
\[|S| \geq 2 \dim(\M_p^n) - \sdim(\mathcal{S}_{G,\M_2^n}) = 2np^2 - ((n+1)p^2 - p) = (n-1)p^2+p.\]
\end{proof}

In particular these results show that  $\C[\M^n_3]^G$ contains no polynomial or hypersurface separating set for $n \geq 3$, and that $\C[\M_p^n]^G$ contains no polynomial or hypersurface separating set for $p \geq 4$ and $n \geq 2$. Let $Q_n$ be a quiver with one node and $n$ loops, and recall that $Q_n$ has finite representation type if $n=1$ and wild representation type otherwise. We thus obtain the following result:

\begin{Theorem}\label{conjproved}
 The representation type of $Q_n$ is tame or finite if and only if for all dimension (vectors) $p$, $\C[\mathrm{Rep}(Q_n,p)]^{\SL_p(\C)}$ contains a polynomial or hypersurface separating set.
\end{Theorem}

\begin{proof}
When $n=1$ it is well-known that $\C[\M^n_p]^G$ is a polynomial ring for all $p$. Conversely, by Theorem \ref{mainsepvar} $\C[\M^n_4]^G$ contains no polynomial or hypersurface separating set for all $n \geq 2$. 
\end{proof}

In other words, Conjecture \ref{conj} is indeed true for the $n$-loop quiver.

\begin{rem} If we were interested in finding the smallest $p$ such that $\C[\mathrm{Rep}(Q_n,p)]^{\SL_p(\C)}$ contains no polynomial or hypersurface separating set, we could take $p=2$ when $n \geq 4$ and $p=3$ when $n=3$, and $p=4$ when $n=2$.
\end{rem}

\subsection{Organisation of paper}
This paper is organised as follows: in Section 2 we describe a coarse decomposition of the separating variety for $S_{G,\M_p^n}$ for any $p \geq 2$ and $n \geq 2$. In Section 3 we prove that the irreducible components of $S_{G,\M_p^n}$ are in 1-1 correspondence with the maximal elements of 
 a certain poset $\mathcal{P}_{p,n}$ which we describe in a purely combinatorial fashion. In Section 4 we characterise the maximal elements of $\mathcal{P}_{p,n}$ when $n \geq 3$. This allows us to give a formula for the number of components of  $S_{G,\M_p^n}$ of given codimension, and prove Theorem \ref{mainsepvar} for $n \geq 3$. We also give explicit irredundant decompositions of  $S_{G,\M_p^n}$ when $p=2,3,4$ and $n \geq 3$. In Section 5 we prove Theorem \ref{mainsepvar} for $n=2$ and give explicit irredundant decompositions of  $S_{G,\M_p^2}$ when $p=2,3,4$. In Section 6 we use a homomorphism due to Domokos to deduce lower bounds for the cardinality of a separating set for matrix semi-invariants and prove Conjecture \ref{conj} for the $n$-Kronecker quiver.

\begin{ack}
This work began while the author was visiting Prof. Harm Derksen at Northeastern University, a visit made possible by the EPSRC small grant scheme (ref: EP/W001624/1). The author would like to thank the research council for their support, and Prof. Derksen for his hospitality and help with the computations which inspired this work. The author would also like to thank Lieven Le Bruyn, Vesselin Drensky and Gregor Kemper for helpful answers to questions. 
\end{ack}

\section{Decomposing the separating variety for matrix invariants}\label{sec:sepvarmatinv}

In this section we consider the action of $G:= \GL_p(\C)$ on $\vv:= \M_{p}^n$ by simultaneous conjugation, where $n \geq 2$. Our aim is to describe the separating variety $\mathcal{S}_{G,\vv}$ as a union of closed irreducible subvarieties. 

For each $k \leq p$ let $\Pi(p,k)$ denote the set of ordered partitions of $p$ into $k$ nonzero parts. If $\pi \in \Pi(p,k)$ we say $k = |\pi|$ is the rank of $\pi$. We define $\Pi(p) = \sqcup_{k=1}^p \Pi(p,k)$. We write as a convention $\pi = (p_1,p_2, \ldots, p_k) \in \Pi(p,k)$, and if we need to discuss a second partition $\hat{\pi}$ we write $|\hat{\pi}| = \hat{k}$ and $\hat{\pi} = (\hat{p}_1,\hat{p}_2, \ldots, \hat{p}_{\hat{k}}) \in \Pi(p,\hat{k})$. 
Let $\pi \in \Pi(p,k)$, and $\hat{\pi} \in \Pi(p,k-1)$. We write $\pi \leq \hat{\pi}$ if and only if there exists some $1 \leq j \leq k-1$ with the properties that
\begin{itemize}
\item $p_i = \hat{p}_i$ for all $i<j$;
\item $p_j+p_{j+1} = \hat{p}_j$
\item $p_i = \hat{p}_{i-1}$ for all $i>j$.
\end{itemize}
In other words if $\hat{\pi}$ is formed by merging two blocks of $\pi$. We extend the relation $\leq$ to all pairs of partitions in $\Pi(p)$ by taking the transitive and reflexive closure. In this way $\Pi(p)$ becomes an anitgraded poset with respect to rank and $\leq$. We say $\pi$ is a refinement of $\hat{\pi}$ if $\pi \leq \hat{\pi}$.

For $\pi$ and $\hat{\pi}$ partitions of $p$, a {\it  $(\pi,\hat{\pi})$-block $n$-matrix} is an $n$-matrix partitioned by $\pi$ along the rows and $\hat{\pi}$ along the columns, i.e. into rectangular $p_i \times \hat{p}_{j}$ submatrices, and a  {\it $\pi$-block $n$-matrix} is a $(\pi, \pi)$-block $n$ matrix. We adopt the convention that if $\bA \in \vv$ is a $(\pi,\hat{\pi})$-block $n$-matrix, then $\bB_{i,j}$ denotes the $p_i \times \hat{p}_j$ $n$-matrix in the $(i,j)$th block, with a similar convention for a second $n$-matrix $\bA'$. For an ordered $k$-tuple of matrices $(W_i \in \M_{p_i} :i=1, \ldots, k)$ we write $\bigoplus_{i=1}^k W_i$ for the $\pi$-block diagonal matrix with $(i,i)$th block equal to $W_i$.

 Let $\pi \in \Pi(p,k)$. We denote by $\vv_{\pi}$ the subspace of $\vv$ whose elements are $\pi$-block-upper triangular $n$-matrices, i.e. those $\bA \in \vv$ with $\bB_{i,j} = 0$ for $i>j$. The dimension of this space is
\begin{equation}\label{dimVpi}
\frac12 n (p^2+\sum_{i=1}^k p_i^2).
\end{equation}

The easiest way to see this is to consider the opposite subspace $\tilde{\vv}_{\pi}$ of $\pi$-block-lower triangular matrices. Evidently $\vv_{\pi}+\tilde{\vv}_{\pi} = \vv$, and $\dim(\vv_{\pi}) = \dim(\tilde{\vv}_\pi)$, while $\dim(\vv_\pi \cap \tilde{\vv}_{\pi}) = n(\sum_{i=1}^k p_i^2)$ as this is the space of $\pi$-block diagonal $n$-matrices.

 We denote by $P_{\pi}$ the parabolic subgroup of $G$ stabilising $\vv_{\pi}$. By an argument similar to the above, its dimension is $\frac12  (p^2+\sum_{i=1}^k p_i^2)-1$. The following Lemma generalises \cite[Lemma~3.4]{ElmerSemiSepVar}:

\begin{Lemma}\label{harm}
Let $G$ be a reductive linear algebraic group over $\kk$ and $P \leq G$ a parabolic subgroup. Let $V$ be a vector space over $\kk$ on which $G$ acts linearly and let $W$ be a closed subset of $V$. Suppose $P \cdot W \subseteq W$. Then $G \cdot W$ is closed.
\end{Lemma}  

\begin{proof} Consider the set $Z:= \{(gP,v): g^{-1}v \in W\} \subseteq G/P \times V$. Let $\pi$ be the projection $G/P \times V \rightarrow V$. Then $G \cdot W = \pi(Z)$. As $G/P$ is a projective, hence complete, variety, the map $\pi$ is projective and takes closed sets to closed sets. Clearly $Z$ is closed, so $G \cdot W$ is closed as required.
\end{proof}

Consequently each orbit space $G \cdot \vv_{\pi}$ is closed, with dimension equal to 
\begin{equation}\label{dimGVpi} \dim(G)-\dim(P_{\pi})+\dim(\vv_{\pi})=p^2 + \frac12 (n-1) (p^2+\sum_{i=1}^k p_i^2).
\end{equation} 

For a permutation $\sigma \in S_k$ we set $\sigma(\pi) = (p_{\sigma(1)},p_{\sigma(2)}, \ldots, p_{\sigma(k)}) \in \Pi(p,k)$.  The action of $S_k$ induces an equivalence relation on $\Pi(p,k)$. We write $[\pi]$ for the equivalence class containing $\pi$. We write $$[\Pi(p,k)] = \Pi(p,k)/S_k$$ and $[\Pi(p)] = \sqcup_{k=1}^p [\Pi(p,k)]$. The partial order on $\Pi(p)$ induces a partial order on $[\Pi(p)]:$ we write $[\pi] \leq [\hat{\pi}]$ if and only if $\kappa \leq \hat{\kappa}$ for some $\kappa \in [\pi]$ and $\hat{\kappa} \in [\hat{\pi}]$. 

Now for each $\pi \in \Pi(p,k)$ and for each permutation $\sigma \in S_k$, we define a subspace $\cc_{\pi,\sigma}$ of $\vv^2$ as follows: $\cc_{\pi,\sigma}$ consists of pairs of $n$-matrices $(\bA,\bA')$ such that:

\begin{itemize}
\item $\bA \in \vv_{\pi}$. 
\item $\bA' \in \vv_{\sigma(\pi)}$.
\item $\bB_{i,i} = \bB'_{\sigma(i),\sigma(i)}$ for all $i=1,\ldots, k$. 
\end{itemize}

The dimension of $\cc_{\pi, \sigma}$ is $n(p^2+ \sum_{i=1}^k p_i^2)$. Notice that $G^2 \cdot \cc_{(p),\id} = \{(\bA, g \cdot \bA): \bA \in \vv, g \in G\}$ is the graph of the action. It follows from Proposition \ref{closures} that we have a decomposition of the separating variety into closed irreducible sets:
\begin{equation}\label{fundamentaldecomp}
\mathcal{S}_{G,\vv} = \bigcup_{k=1}^p \bigcup_{\pi \in \Pi(p,k)} \bigcup_{\sigma \in S_k} \overline{G^2 \cdot \cc_{\pi,\sigma}}.
\end{equation}
This decomposition may be redundant, i.e. we may be able to delete some terms from the right hand side. Note that, since each term on the right hand side is irreducible, redundancy is only possible if there are inclusions between the terms; if $\cc \subseteq \mathcal{A} \cup \mathcal{B}$ and $\cc$ is irreducible then $\cc \subseteq \mathcal{A}$ or $\cc \subseteq \mathcal{B}$.

\subsection{Dimensions}

Our first task is to compute the dimensions of the varieties on the right hand side of \eqref{fundamentaldecomp}. Note it is not clear which, if any, of the orbit sets $G^2 \cdot \cc_{\pi,\sigma}$ are closed.   

\begin{cor}\label{speculation}
We have \[\overline{G^2 \cdot \cc_{\pi, \sigma}} = G^2 \cdot \overline{(P_\pi \times P_{\sigma(\pi)}) \cdot \cc_{\pi, \sigma}}.\]
In particular $G^2 \cdot \cc_{\pi, \sigma}$ is closed if $\pi = (1,1, \ldots, 1)$.
\end{cor}


\begin{proof} The first part follows on applying Lemma \ref{harm} to $\overline{(P_\pi \times P_{\sigma(\pi)}) \cdot \cc_{\pi, \sigma}}$ which is obviously stablilised by the parabolic subgroup $P_{\pi} \times P_{\sigma(\pi)}$ of $G^2$. For the second statement we note that if $\pi = (1, \ldots, 1)$, then $\sigma(\pi) = \pi$ and $P _{\pi}\times P_{\sigma(\pi)} = P^2_{\pi}$ stabilises $\cc_{\pi, \sigma}$.
\end{proof}

The varieties $\overline{(P_\pi \times P_{\sigma(\pi)}) \cdot \cc_{\pi, \sigma}}$ have the following inductive description.

\begin{Lemma}\label{clCpisigma} Let $(\bA, \bA') \in \vv^2$. Then $(\bA, \bA') \in \overline{(P_\pi \times P_{\sigma(\pi)}) \cdot \cc_{\pi, \sigma}}$ if and only if $\bA \in \vv_\pi, \bA' \in \vv_{\sigma(\pi)}$ and $(\bB_{i,i},\bB'_{\sigma(i), \sigma(i)}) \in \overline{\GL_{p_i}^2 \cdot \cc_{(p_i), \id}}$ for all $i=1, \ldots, k$.
\end{Lemma} 

\begin{proof} The necessity of $\bA \in \vv_\pi, \bA' \in \vv_{\sigma(\pi)}$ is clear as $P_\pi$ fixes $\vv_{\pi}$ and $P_{\sigma(\pi)}$ fixes $\vv_{\sigma(\pi)}$. Suppose that $(\bB_{i,i},\bB'_{\sigma(i), \sigma(i)}) \in \overline{\GL_{p_i}^2 \cdot \cc_{(p_i), \id}}$ for all $i=1, \ldots, k$. Then there exist $g_i(t): \in \GL_{p_i}(\C((t))$, $\bC_{i,i}(t) \in \M_{p_i,p_i}(\C[[t]])$ such that 
\[\lim_{t \rightarrow 0} \bC_{i,i}(t) = \bB_{i,i}\]
and 
\[\lim_{t \rightarrow 0} g_i(t) \cdot \bC_{i,i}(t) = \bB'_{\sigma(i),\sigma(i)}\] for all $i=1,\ldots, k$.
Let $g(t) \in P_{\pi}(\C((t)))$ be defined by
 $$g(t):= \bigoplus_{i=1}^k g_i(t) \in P_{\pi},$$ i.e. $g(t)$ is the block-diagonal matrix with blocks given by $g_i(t)$. Now we define $\bA \in \vv(\C[[t]])$ blockwise according to the partition $\pi$: the image of the $i,j$th block is given by
\[\bB_{i,j}(t) = \left\{ \begin{array}{lr} \bC_{i,j}(t) & i=j \\ \bB_{i,j} & \ \text{otherwise}. \end{array} \right.\] 

Then we have 
$$(g(t) \cdot \bA(t), \bA'(t)) \in \cc_{\pi, \sigma}$$ for all $t \in \C^*$. It follows that $(\bA,\bA') \in \overline{(P_{\pi} \times P_{\sigma(\pi)}) \cdot  \cc_{\pi, \sigma}}$ as required.

Conversely suppose that $(\bA,\bA') \in \overline{(P_{\pi} \times P_{\sigma(\pi)}) \cdot  \cc_{\pi, \sigma}}$. Then there exist $g(t) \in P_{\pi}(\C((t)))$, $\bA(t) \in \vv_{\pi}(\C[[t]])$ and  $\bA'(t) \in \vv_{\sigma(\pi)}(\C[[t]])$  such that 
\[\lim_{t \rightarrow 0} \bA(t) = \bA, \lim_{t \rightarrow 0} \bA'(t) = \bA'\]
and \begin{equation}\label{above} (g(t) \cdot \bA(t), \bA'(t)) \in \cc_{\pi, \sigma}\end{equation} for all $t \in \C^*$.
Let $g_i(t) \in \GL_{p_i}$ denote the $(i,i)th$ block of $g(t)$. Let $\bB_{i,j}(t), \bB'_{i,j}(t)$ denote the $i,j$th block of $\bA(t), \bA'(t)$ respectively.
\eqref{above} above shows that $$g_i(t) \cdot \bB_{i,i}(t) = \bB'_{\sigma(i),\sigma(i)}(t)$$ for all $i = 1, \ldots, k$ and all $t \in \C^*$. Therefore  $(\bB_{i,i},\bB'_{\sigma(i), \sigma(i)}) \in \overline{\GL_{p_i}^2 \cdot \cc_{(p_i), \id}}$ for all $i=1, \ldots, k$ as required.    
\end{proof}

\begin{cor}\label{dimG2Cpisigma}
The dimension of the variety $\overline{G^2 \cdot \cc_{\pi, \sigma}}$ is $(n+1) p^2 - k$.
\end{cor}

\begin{proof} First suppose $k=1$, then $\pi = (p)$ and $\sigma = \id$, and $G^2 \cdot \cc_{\pi, \sigma}$ is the graph of the action, hence
\[\dim(\overline{G^2 \cdot \cc_{\pi, \sigma}}) = \dim(\overline{G^2 \cdot \cc_{(p), \id}}) = \dim(\vv) + \dim(G) = np^2+p^2-1 = (n+1)p^2 -1\]
as required. Now suppose $k>1$, then by Corollary \ref{speculation} we have

\begin{align*}\dim(\overline{G^2 \cdot  \cc_{\pi, \sigma}}) &= \dim(G^2 \cdot \overline{(P_{\pi} \times P_{\sigma(\pi)}) \cdot \cc_{\pi, \sigma}})\\
&= \dim(G^2) - \dim(P_{\pi}) - \dim(P_{\sigma(\pi)}) + \dim(\overline{(P_{\pi} \times P_{\sigma(\pi)}) \cdot \cc_{\pi, \sigma}})\\
&= 2\dim(G) - 2\dim(P_{\pi})+ \dim(\overline{(P_{\pi} \times P_{\sigma(\pi)}) \cdot \cc_{\pi, \sigma}})\\
&=2(p^2-1) - (p^2-2+ \sum_{i=1}^k p_i^2)+ \dim(\overline{(P_{\pi} \times P_{\sigma(\pi)}) \cdot \cc_{\pi, \sigma}})\\
&=p^2 - \sum_{i=1}^k p_i^2+ \dim(\overline{(P_{\pi} \times P_{\sigma(\pi)}) \cdot \cc_{\pi, \sigma}}).
\end{align*}

Now by Lemma \ref{clCpisigma} we deduce that
\begin{align*} \dim(\overline{(P_{\pi} \times P_{\sigma(\pi)}) \cdot \cc_{\pi, \sigma}}) &= (\dim(\cc_{\pi, \sigma}) - 2n (\sum_{i=1}^k p_i^2)) + \sum_{i=1}^k \dim(\overline{\GL^2_{p_i} \cdot \cc_{(p_i), \id}}) \\
&=   n (p^2 +\sum_{i=1}^k p_i^2) - 2n \sum_{i=1}^k p_i^2 + \sum_{i=1}^k ((n+1)p_i^2 - 1)\\
\intertext{by the $k=1$ case}
&=  n (p^2 -\sum_{i=1}^k p_i^2) +  \sum_{i=1}^k ((n+1)p_i^2 - 1)\\
&= np^2 + \sum_{i=1}^k p_i^2 - k\\
\intertext{hence}
\dim(\overline{G^2 \cdot \cc_{\pi, \sigma}}) &= (n+1)p^2-k \end{align*}
as required.
\end{proof}

We note that we have proved, using the above and \eqref{fundamentaldecomp},

\begin{cor} The separating variety $\mathcal{S}_{G,\vv}$ has dimension $(n+1)p^2-1$ and subdimension $\geq (n+1)p^2-p$.
\end{cor}

\subsection{Distinctness and compatibility}
Next we want to prove that the varieties on the right hand side of \eqref{fundamentaldecomp} are distinct. We begin by showing: 
\begin{Lemma}\label{inclusionwithsimples} Let $\pi \in \Pi(p,k)$ and $\hat{\pi} \in \Pi(p,\hat{k})$. Suppose we have an inclusion $G \cdot \vv_{\pi} \subseteq G \cdot \vv_{\hat{\pi}}$. Then $[\pi] \leq [\hat{\pi}]$.
\end{Lemma}

\begin{proof} By \eqref{dimGVpi} we have $k \geq \hat{k}$. Let $\bA \in \vv_{\pi}$ be $\pi$-block diagonal with simple blocks. Since $\bA \in G \cdot \vv_{\hat{\pi}}$, $\overline{G \cdot \bA}$ contains a unique closed orbit $\overline{G \cdot \hat{\bA}}$ with $\hat{\bA} \in \vv_{\hat{\pi}}$ $\hat{\pi}$-block diagonal. As $\bA$ is semisimple, $G \cdot \bA = G \cdot \hat{\bA}$ and therefore $\pi \leq \sigma(\hat{\pi})$ for some $\sigma \in S_{\hat{k}}$. 
\end{proof}

For the next result we need a new concept. Let $\bB \in \M_p$ and $\bB' \in \M_{p'}$. The map $d_{\bB,\bB'}$ from Proposition \ref{ext} has domain $\M_{p,p'}$ and codomain $\M^n_{p,p'}$, so for $n \geq 2$ cannot be surjective. Therefore by Proposition \ref{ext}, there exists a nontrivial extension of $\bB$ and $\bB'$.

Now let $\bA \in \vv_{\pi}$. We will say that $\bA$ is {\it maximally general} (for $\pi$) if the following hold:

\begin{itemize}
\item The diagonal $\pi$-blocks $\bB_{i,i}$ of $\bA$ are simple and pairwise non-isomorphic;

\item For each $i=1, \ldots, k-1$, the submatrix
\[\bC_i = \begin{pmatrix} \bB_{i,i} & \bB_{i,i+1} \\ 0 & \bB_{i+1,i+1} \end{pmatrix}\]
is indecomposable as a representation of $\GL_{p_i+p_{i+1}}$.
\end{itemize}

It is a consequence of the discussion above that maximally general matrices exist in $\vv_{\pi}$ for arbitrary choices of simple diagonal blocks. Moreover, maximal generality is preserved by the action of $P_{\pi}$.

Maximally general matrices have the following useful property: 
\begin{Lemma}\label{maxgen} Let $\bA \in \vv_{\pi}$ be maximally general. Let $\sigma \in S_k$ be non-trivial. Then the orbit of $\bA$ does not contain any element $\bA' \in \vv_{\sigma(\pi)}$ with diagonal blocks $\bB'_{i,i} = \bB_{\sigma^{-1}(i), \sigma^{-1}(i)}.$
\end{Lemma}

The proof requires some preparation. Identify the Weyl group $S_p$ of $G$ with the subgroup of permutation matrices. Then $B:= P_{1,1, \ldots, 1}$ is the Borel subgroup with respect to the choice of $\{(1,2), (2,3), \ldots, (p-1,p)\}$ as fundamental roots. For a partition $\pi \in \Pi(p,k)$ we define $\pi_1 = \{1, \ldots, p_1\}$,
\[\pi_i= \{p_1+ \ldots+p_{i-1}+1, \ldots, p_1+ \ldots + p_i\}\]
for $i=2. \ldots, k$ and
\[S_{\pi} = \prod_{i=1}^k S_{\pi_i} \leq S_p.\]
This is the Young subgroup of $S_p$ associated with $\pi$, and we have Bruhat decomposition.
\[P_{\pi} = BS_{\pi}B.\]

\begin{proof}[Proof of Lemma \ref{maxgen}]
Suppose the contrary, and write $g \cdot \bA = \bA' \in \vv_{\pi'}$, where $\pi' = \sigma(\pi)$. Now $G$ has Bruhat decomposition
\[G = \bigcup_{t \in T_{\pi', \pi}} P_{\pi'}tP_{\pi},\] where $T_{\pi',\pi}$ is a set of $S_{\pi'}-S_{\pi}$ double coset representatives in $S_p$. Therefore we may write $g  = h'th$ for some $h' \in P_{\pi'}$, $t \in T_{\pi',\pi}$ and $h \in P_{\pi}$. We therefore see that
\[t \cdot (h \cdot \bA) = (h')^{-1} \cdot \bA'.\]
Now $(h')^{-1} \cdot \bA' \in \vv_{\pi'} $ has diagonal blocks isomorphic to those of $\bA'$. Further,  $h \cdot \bA$ has diagonal blocks isomorphic to those of $\bA$ and is maximally general. In particular, its superdiagonal blocks (i.e. those in position $(i,i+1)$ for some $i=1, \ldots, k-1$) are nonzero.

As the diagonal blocks of $h \cdot \bA$ and $(h')^{-1} \cdot \bA'$ are simple and non-isomorphic, we must have $t(\pi_i) = \pi'_i = \pi_{\sigma(i)}$ for each $i$. Since $\sigma$ is nontrivial, there exists some $1 \leq i <k$ with $\sigma(i)>\sigma(i+1)$. But then the $(\sigma(i), \sigma(i+1))$th $\pi'$-block in $(h')^{-1} \cdot \bA'$ is isomorphic to the nonzero $(i,i+1)$th $\pi$-block in $h \cdot \bA$, contradicting $(h') \cdot \bA' \in \vv_{\pi'}$.  
\end{proof}

\begin{cor}\label{ordercor} We have an inclusion $G \cdot \vv_{\pi} \subseteq G \cdot \vv_{\hat{\pi}}$ if and only if $\pi \leq \hat{\pi}$.
\end{cor}

\begin{proof} The if direction is clear. For the only if direction, choose $\bA \in \vv_{\pi}$ maximally general and suppose $\bA \in G \cdot \vv_{\hat{\pi}}$. Then there exists $g \in G$ such that $\hat{\bA}:= g \cdot \bA \in \vv_{\hat{\pi}}$. If $|\hat{\pi}| = \hat{k}<k$ then not all $\hat{\pi}$-blocks of $\hat{\bA}$ can be simple. For each $i=1, \ldots, \hat{k}$ we may find a partition $\mu_i$ of $\hat{p}_i$ and an element $h_i \in \GL_{\hat{p}_i}$ such that $h_i \cdot \hat{\bB}_{i,i} \in \vv_{\mu_i}$ with simple diagonal blocks. Let $h = \bigoplus_{i=1}^{\hat{k}} h_i$, and let $\pi' \in \Pi(p,k')$ be the partition of $p$ formed by dividing each $\hat{p}_i$ according to $\mu_i$. Then $\bA':= h \cdot \hat{\bA} \in \vv_{\pi'}$ has simple diagonal $\pi'$-blocks. On the other hand if $k= \hat{k}$ then all $\hat{\pi}$-blocks of $\hat{\bA}$ are simple and we may take $h=1$ and $\pi'=\hat{\pi}$. Notice that $\bA' = (hg \cdot \bA)$, so by uniqueness of representation type we must have $k=k'$ and $\pi' =\sigma(\pi)$ for some $\sigma \in S_k$, and ${\bB'}_{i,i} = \bB_{\sigma^{-1}(i), \sigma^{-1}(i)}$ for all $i$. Applying Lemma \ref{maxgen} we have a contradiction unless $\sigma$ is trivial, so we must have $\pi = \pi' \leq \hat{\pi}$ as required.
\end{proof}

Regarding the spaces $\overline{G^2 \cdot \cc_{\pi, \sigma}}$, we have the following:

\begin{Lemma}\label{reducer}
Let $(\bA, \bA') \in \cc_{\pi, \sigma}$ with both $\bA$ and $\bA'$ maximally general. Suppose $(\bA, \bA') \in \overline{G^2 \cdot \cc_{\hat{\pi}, \hat{\sigma}}}$. Then:
\begin{enumerate}
\item[(a)]  $\pi \leq \hat{\pi}$ and $\sigma(\pi) \leq \hat{\sigma}(\hat{\pi})$. 

\item[(b)] Let $\hat{\bB}_{i,j}$ be the $\hat{\pi}$ blocks of $\bA$ and let $\hat{\bB'}_{i,j}$ be the $\hat{\sigma}(\hat{\pi})$ blocks of $\bA'$. Then
\[(\hat{\bB}_{i,i}, \hat{\bB'}_{\hat{\sigma}(i), \hat{\sigma}(i)}) \in \overline{\GL_{\hat{p}_i}^2 \cdot \cc_{(\hat{p}_i), \id}}\] for all $i=1, \ldots, \hat{k}$.
\end{enumerate}
\end{Lemma}

\begin{proof} (a) follows on applying the argument of Corollary \ref{ordercor} to $\bA$ and $\bA'$. For (b) we note that by Corollary \ref{speculation} there exist $g, g' \in G$ such that $(g \cdot \bA, g' \cdot \bA') \in \overline{(P_{\hat{\pi}} \times P_{\hat{\sigma}(\hat{\pi})}) \cdot \cc_{\hat{\pi},\hat{\sigma}}}$. Let $\tilde{\bB}_{i,j}, \tilde{\bB'}_{i,j}$ denote the $\hat{\pi}$- blocks, resp. $\hat{\sigma}(\hat{\pi})$-blocks of $g \cdot \bA,g' \cdot \bA'$ respectively. By Lemma \ref{clCpisigma} we have $(\tilde{\bB}_{i,i}, \tilde{\bB'}_{\hat{\sigma}(i), \hat{\sigma}(i)}) \in \overline{\GL_2^2 \cdot \cc_{(\hat{p}_i), \id}}$ for all $i=1, \ldots, \hat{k}$. Now as $\bA$ is maximally general for $\pi$, by Lemma \ref{maxgen} we must have $g \in P_{\pi} \leq P_{\hat{\pi}}$. Similarly $\bA'$ is maximally general for $\sigma(\pi)$ and so $g' \in P_{\sigma(\pi)} \leq P_{\hat{\sigma}(\hat{\pi})}$. Let $h_1, h_2, \ldots, h_{\hat{k}}$, resp. $h'_1, \ldots, h'_{\hat{k'}}$ be the $\hat{\pi}$, resp. $\hat{\sigma}(\hat{\pi})$ diagonal blocks of $g^{-1}$ resp. $(g')^{-1}$. Then
\[ (\hat{\bB}_{i,i}, \hat{\bB'}_{\hat{\sigma}(i), \hat{\sigma}(i)})  = (h_i \cdot \tilde{\bB}_{i,i}, h'_i \cdot \tilde{\bB'}_{\hat{\sigma}(i), \hat{\sigma}(i)}) \in  \overline{\GL_{\hat{p}_i}^2 \cdot \cc_{(\hat{p}_i), \id}}.\]
\end{proof}

We may now deduce that the spaces on the right hand side of \eqref{fundamentaldecomp} are distinct:

\begin{cor}\label{distinct}
We have $\overline{G^2 \cdot \cc_{\pi, \sigma}} = \overline{G^2 \cdot \cc_{\hat{\pi}, \hat{\sigma}}}$ if and only if $\pi = \hat{\pi}$ and $\sigma = \hat{\sigma}$.
\end{cor}

\begin{proof} The if direction is clear. Let $\pi \in \Pi(p,k)$ and $\hat{\pi} \in \Pi(p,\hat{k})$ and assume $\overline{G^2 \cdot \cc_{\pi, \sigma}} = \overline{G^2 \cdot \cc_{\hat{\pi}, \hat{\sigma}}}.$ By Corollary \ref{dimG2Cpisigma} we must have $k = \hat{k}$. By Lemma \ref{reducer}(a) we have $\pi = \hat{\pi}$. Now let $(\bA, \bA') \in \cc_{\pi, \sigma}$ such that $\bA$ and $\bA'$ are maximally general, and suppose  $(\bA, \bA') \in \overline{G^2 \cdot \cc_{\pi, \hat{\sigma}}}$. Then 
 \[(\bB_{i,i}, \bB'_{\sigma(i), \sigma(i)}) \in  \overline{\GL_{p_i}^2 \cdot \cc_{(p_i), \id}}\] for all $i$, and
by Lemma \ref{reducer}(b), 
 \[(\bB_{i,i}, \bB'_{\hat{\sigma}(i), \hat{\sigma}(i)}) \in  \overline{\GL_{p_i}^2 \cdot \cc_{(p_i), \id}}.\]
Since the diagonal blocks of $\bA$ and $\bA'$ are simple, their $\GL_{p_i}$-orbits are closed, therefore
 \[\bB_{i,i} \cong \bB'_{\sigma(i), \sigma(i)} \cong \bB'_{\hat{\sigma}(i), \hat{\sigma}(i)}\] for all $i=1, \ldots, k$.
As the diagonal blocks of $\bA'$ are pairwise non-isomorphic, we must have $\sigma = \hat{\sigma}$ as required.
\end{proof}

Suppose  $\overline{G^2 \cdot \cc_{\pi, \sigma}} \subseteq \overline{G^2 \cdot \cc_{\hat{\pi}, \hat{\sigma}}}$. Then in particular every maximally general pair $(\bA,\bA') \in \cc_{\pi, \sigma}$ belongs to $ \overline{G^2 \cdot \cc_{\hat{\pi}, \hat{\sigma}}}$, and so by the above $\pi \leq \hat{\pi}$ and $\sigma(\pi) \leq \hat{\sigma}(\hat{\pi})$, with \[(\hat{\bB}_{i,i}, \hat{\bB'}_{\hat{\sigma}(i), \hat{\sigma}(i)}) \in \overline{\GL_{\hat{p}_i}^2 \cdot \cc_{(\hat{p}_i), \id}}\] for all $i=1, \ldots, \hat{k}$.

Since $\pi \leq \hat{\pi}$, there exists a partition $(k_1,k_2, \ldots, k_{\hat{k}}) \in \Pi(k,\hat{k})$, where $k_i$ is the number of $\pi$-blocks which combine to make the $i$th block of $\hat{\pi}$. That is,
\[\hat{p}_i = \sum_{j=1}^{k_i} p_{k_1+k_2+ \cdots + k_{i-1}+j}.\]

The simple $\pi$-blocks of $\bA$ are $(\bB_{1,1}, \ldots, \bB_{k,k})$. Therefore the simple $\sigma(\pi)$-blocks of $\bA'$ are $$(\bB'_{1,1}, \ldots, \bB'_{k,k}) = (\bB_{\sigma(1),\sigma(1)}, \ldots, \bB_{\sigma(k),\sigma(k)}).$$
On the other hand, the $\hat{\pi}$-blocks of $\bA$ are $(\hat{\bB}_{1,1}, \ldots, \hat{\bB}_{\hat{k},\hat{k}})$, and the simple blocks of $\hat{\bB}_{i,i}$ are $$(\bB_{k_1+\cdots+k_{i-1}+j, k_1+\cdots+k_{i-1}+j}: j=1, \ldots, k_i).$$

Since in particular $(\hat{\bB}_{i,i}, \hat{\bB'}_{\hat{\sigma}(i), \hat{\sigma}(i)}) \in \mathcal{S}_{\GL_{\hat{p}_i},\M^n_{\hat{p}_i}}$, the simple blocks of  $\hat{\bB'}_{\hat{\sigma}(i), \hat{\sigma}(i)}$ must be the same list up to order. So, there exists a permutation $\tau_i \in S_{k_i}$ such that these are $$(\bB_{k_1+\cdots+k_{i-1}+\tau_i(j), k_1+\cdots+k_{i-1}+\tau_i(j)}: j=1, \ldots, k_i).$$

These must be a set of $k_i$ consecutive $\sigma(\pi)$-blocks of $\bA'$. 

Now, given any $j \leq k$ there exists a unique $l \leq \hat{k}$ and $j' \leq k_{\hat{\sigma}(l)}$ such that
\[j = k_{\hat{\sigma}(1)}+ \cdots + k_{\hat{\sigma}(l-1)}+j'.\]
Since the simple blocks of $\bA'$ are pairwise non-isomorphic we have: 
\begin{equation}\label{compatible first} \sigma(j) = k_1+k_2 \ldots + k_{\hat{\sigma}(l)-1}+\tau_{l}(j').\end{equation}
(Here, $\sigma(0)$ is to be interpreted as $0$, so that $j=j'$ if $l=1$.)

We will say the pairs $(\pi, \sigma)$ and $(\hat{\pi}, \hat{\sigma})$ are {\it compatible} if \eqref{compatible first} holds for some choice of $(\tau_1,\tau_2, \ldots, \tau_{\hat{k}}) \in S_{k_1} \times \cdots \times S_{k_{\hat{k}}}.$ Another way to view this notion is as follows: let $(k_1, k_2, \ldots, k_{\hat{k}}) \in \Pi(k,\hat{k})$. There is a natural action of $S_{\hat{k}}$ on $S_{k_1} \times S_{k_2} \cdots \times S_{k_{\hat{k}}}$ by permuting the arguments. Given $\mathbf{\tau} = (\tau_1, \tau_2, \ldots, \tau_{\hat{k}}) \in S_{k_1} \times S_{k_2} \cdots \times S_{k_{\hat{k}}}$ we write $\mathbf{\tau}^{\hat{\sigma}}$ for the result of applying $\hat{\sigma} \in S_{\hat{k}}$ to $\mathbf{\tau}$. We define $\mathbf{\tau}^* \in S_k$ by first defining $\tau'_i \in S_{\{k_1+k_2+\cdots+k_{i-1}+1, \ldots, k_1+k_2+\cdots+k_i\}}$ by
\[\tau'_i(j) = k_1+k_2+ \cdots+k_{i-1}+\tau_i(j)\] and setting $\mathbf{\tau}^* = \tau'_1\tau'_2 \cdots \tau'_{\hat{k}}$. Now let $\pi \leq \hat{\pi}$ and let $\sigma \in S_k$ and $\hat{\sigma} \in S_{\hat{k}}$. Let $k_i$ be the number of blocks of $\pi$ which are combined to make the $i$th block of $\hat{\pi}$.  
Then the pairs $(\pi, \sigma)$ and $(\hat{\pi}, \hat{\sigma})$ are compatible if and only if
\begin{equation}\label{compatible} \sigma = (\mathbf{\tau}^{\hat{\sigma}})^* \end{equation}
for some $\mathbf{\tau} \in S_{k_1} \times S_{k_2} \cdots \times S_{k_{\hat{k}}}$.

In light of the discussion above, we have shown the following:

\begin{prop}\label{compatibility} If $\overline{G^2 \cdot \cc_{\pi, \sigma}} \subseteq \overline{G^2 \cdot \cc_{\hat{\pi}, \hat{\sigma}}}$ then the pairs 
$(\pi, \sigma)$ and $(\hat{\pi}, \hat{\sigma})$ are compatible. Moreover, the converse holds if in addition we have, for all $i=1, \ldots, \hat{k}$,
\[\overline{\GL_{\hat{p}_i}^2 \cdot \cc_{\kappa_i, \tau_i}} \subseteq \overline{\GL_{\hat{p}_i}^2 \cdot \cc_{(\hat{p}_i), \id}}\]
where $\kappa_i \in \Pi(\hat{p}_i, k_i)$ is the partition of $\hat{p}_i$ induced by $\pi$.  
\end{prop} 

We will see in the next section that the additional condition is also necessary.

\section{The Poset $\mathcal{P}_{p,n}$}

In this section let $p>1, n > 1$ and define 
\begin{equation}\label{poset} \mathcal{P}_{p,n} = \bigsqcup_{k=1}^p (\Pi(p,k) \times S_k).\end{equation} 
We may rephrase Equation \eqref{fundamentaldecomp} as follows:

\begin{equation}\label{fundamentaldecompposetversion}
\mathcal{S}_{G,\vv} = \bigcup_{(\sigma, \pi) \in \mathcal{P}_{p,n}} \overline{G^2 \cdot \cc_{\pi,\sigma}}.
\end{equation}

Define a partial order on $\mathcal{P}_{p,n}$ by setting $(\pi, \sigma) \preceq (\hat{\pi},\hat{\sigma})$ if and only if $\overline{G^2 \cdot \cc_{\pi, \sigma}} \subseteq \overline{G^2 \cdot \cc_{\hat{\pi}, \hat{\sigma}}}$. Then $ \overline{G^2 \cdot \cc_{\pi,\sigma}}$ is an irreducible component of $\mathcal{S}_{G,\vv}$ if and only if $(\pi, \sigma)$ is maximal in $\mathcal{P}_{p,n}$. We define the rank of $(\pi, \sigma)$ to be $|\pi|$. The key result is the following:

\begin{Theorem}\label{saturated} The poset $\mathcal{P}_{p,n}$ is antigraded with respect to rank, i.e. if $(\pi, \sigma) \preceq (\hat{\pi}, \hat{\sigma})$ then there is a chain 
\[(\pi, \sigma) = (\pi_0, \sigma_0) \preceq (\pi_1,\sigma_1) \preceq \cdots \preceq (\pi_{k-\hat{k}}, \sigma_{k - \hat{k}}) = (\hat{\pi}, \hat{\sigma})\]
with $|\pi_i| = k-i$.
\end{Theorem}

The proof of this result is quite involved and will take up the rest of this section. We begin by showing that some elements of $\mathcal{P}_{p,n}$ are not maximal. 

\begin{Lemma}\label{inclusions}
Let $\pi:=(p_1,p_2) \in \Pi(p,2)$ and let $\sigma = (12) \in S_2$. Then $(\pi, (12)) \preceq ((p), \id)$.
\end{Lemma}

For the proof we need some additional notation.  For a pair $\pi \in \Pi(p,k), \sigma \in S_k$ we let $s_{\pi,\sigma}$ denote the $(\pi, \sigma(\pi))$-block matrix with blocks
\begin{equation}\label{defs} t_{i,j} = \left\{ \begin{array}{lr} I_{p_i} & \ j = \sigma(i) \\ 0 & \ \text{otherwise.} \end{array} \right.\end{equation}

\begin{proof} It is enough to show that  $\cc_{\pi, \sigma} \subseteq \overline{G^2 \cdot \cc_{(p), \id}}$. Let $(\bA,\bA') \in \cc_{\pi, \sigma}$. We write
\[\bA = \begin{pmatrix} \bB_{1,1} & \bB_{1,2} \\ 0 & \bB_{2,2} \end{pmatrix} \in \vv_{\pi}, \bA' = \begin{pmatrix} \bB_{2,2} & \bB'_{1,2} \\ 0 & \bB_{1,1} \end{pmatrix} \in \vv_{\sigma(\pi)}.\]
For each $t \in \C^*$ we define 
\[ \bA(t) = \begin{pmatrix} \bB_{1,1} & \bB_{1,2} \\ t \bB'_{1,2} & \bB_{2,2} \end{pmatrix},\]  and also $g(t) = t I_{p_1} \oplus I_{p_2} \in \GL_{p}$. We note that $\lim_{t \rightarrow 0} \bA(t) = \bA$ and that $\lim_{t \rightarrow 0} g(t) \cdot \bA(t) = s_{\pi, \sigma} \cdot \bA'$. Therefore $$(\bA,s_{\pi, \sigma} \cdot \bA') \in \overline{\GL^2_{p} \cdot \cc_{(p), \id}},$$ implying $(\bA,\bA') \in \overline{\GL^2_{p} \cdot \cc_{(p), \id}}$ as required. 
\end{proof}

For $\sigma \in S_k$ we will use square brackets for the description of $\sigma$ in array notation omitting the first row, so we will write
\[\sigma = [\sigma(1), \sigma(2), \ldots, \sigma(k)],\]
while parentheses will be used to express permutations in disjoint cycle notation.
For later use we note the following consequence of the lemma above:

\begin{cor}\label{reverses} Let $\pi \in \Pi(p,k)$ and let $\sigma = [k,k-1,\ldots,1]$. Then $(\pi, \sigma) \preceq ((p), \id)$.
\end{cor}

\begin{proof} For each $i=0, \ldots, k-1$ let $\pi_i = (p_1+p_2+ \cdots +p_{i+1},p_{i+2},\ldots, p_k) \in \Pi(p,k-i)$ and $\sigma_i = [k-i,k-i-1, \ldots, 1] \in S_{k-i}$. Then for each $i$ we have $\pi_i \leq \pi_{i+1}$ and $\sigma_{i} = (((12), \id, \id, \ldots, \id)^{\sigma_{i+1}})^*$. Further, by Lemma \ref{inclusions} we have $((p_1+p_2+ \cdots + p_{i+1},p_{i+2}),(12)) \preceq ((p_1+\cdots+p_{i+2}, \id)$. Therefore by by Proposition \ref{compatibility} we have inclusions $(\pi_i, \sigma_i) \preceq (\pi_{i+1}, \sigma_{i+1})$ for all $i=0, \ldots, k-2$. Therefore we have chain of inclusions
\[(\pi, \sigma) = (\pi_0, \sigma_0) \preceq (\pi_1,\sigma_1) \preceq \cdots \preceq (\pi_{k-1}, \sigma_{k - 1}) = ((p), \id).\]
\end{proof} 

Our next aim in this section is to show, with one exception, that the inclusions described in Lemma \ref{inclusions} give the only codimension 1 subvarieties in \eqref{fundamentaldecomp} contained in the graph closure.

\begin{prop}\label{newrkcondition} Let $(\bA,\bA') \in \cc_{(p_1,p_2), \id}$. Denote the $(p_1,p_2)$ blocks of $\bA$ and $\bA'$ by $\bB_{i,j}$ and $\bB'_{i,j}$ respectively and note that the diagonal blocks are identical, i.e. $\bB_{1,1}=\bB'_{1,1}$ and $\bB_{2,2}=\bB'_{2,2}$. Assume these blocks are simple. 
 Then the following are equivalent:

\begin{enumerate}
\item[(i)] $(\bA,\bA') \in \overline{G^2 \cdot \cc_{(p), \id}}$;
\item[(ii)] There exist $w,z \in \C$ with $w\bB_{1,2}+z\bB'_{1,2} \in \im(d_{\bB_{1,1}, \bB_{2,2}})$.
\item[(iii)] There exist $X \in \M_{p_1,p_2}$ and $w,z \in \C$, not all zero, such that $w\bB_{1,2}+z\bB'_{1,2} = X\bB_{1,1}-\bB_{2,2}X$.
\end{enumerate}

\end{prop} 

\begin{proof}
(i) $\Rightarrow$ (iii): Suppose first that $(\bA,\bA') \in G^2 \cdot \cc_{(p), \id}$. Then then exists $g \in G$ with $g \cdot \bA = \bA'$. In particular there exists a nonzero $g \in \M_{p}$ with $g \bA = \bA'g$. Since this is a closed condition, it also holds whenever  $(\bA,\bA') \in \overline{G^2 \cdot \cc_{(p), \id}}$. So assume this holds. Let $g_{i,j}$ denote the $(p_1,p_2)$ blocks of $g$, then we have 

\begin{align*}
g_{1,1} \bB_{1,1} - \bB_{1,1}g_{1,1} - \bB'_{1,2}g_{2,1} &= 0;\\
g_{1,1} \bB_{1,2} + g_{1,2}\bB_{2,2} - \bB_{1,1}g_{1,2} - \bB'_{1,2}g_{2,2} &= 0;\\
g_{2,1}\bB_{1,1} - \bB_{2,2,}g_{2,1} &= 0;\\
g_{2,1} \bB_{1,2}+g_{2,2}\bB_{2,2}-\bB_{2,2}g_{2,2} &=0.
\end{align*} 

Since $\bB_{1,1}$ and $\bB_{2,2}$ are simple, the third equation and Schur's Lemma show that either $g_{2,1}$ is invertible and $\bB_{1,1} \cong \bB_{2,2}$, or $g_{2,1}=0$. In the former case we must have $p_1=p_2$, and we obtain (iii) immediately by setting $w=z=0$ and $X = g^{-1}_{2,1}$. In the latter case our equations become

\begin{align*}
g_{1,1} \bB_{1,1} - \bB_{1,1}g_{1,1}&= 0;\\
g_{1,1} \bB_{1,2} + g_{1,2}\bB_{2,2} - \bB_{1,1}g_{1,2} - \bB'_{1,2}g_{2,2} &= 0;\\
g_{2,2}\bB_{2,2}-\bB_{2,2}g_{2,2} &=0.
\end{align*} 

As $\bB_{1,1}$ and $\bB_{2,2}$ are simple we must, again by Schur's Lemma, have $g_{1,1}=wI_{p_1}$ and $g_{2,2}=zI_{p_2}$ for some $w,z \in \C$. Therefore the second equation implies (iii) holds with $X = g_{1,2}$.

(iii) $\Leftrightarrow$ (ii): This is clear from the definition of $d_{\bB_{1,1},\bB_{2,2}}$.

(ii) $\Rightarrow$ (i): if $w, z \neq 0$ then $g \cdot \bA = \bA'$ where $$g = \begin{pmatrix} wI_{p_1} & X \\ 0 & zI_{p_2} \end{pmatrix}.$$
If $w \neq 0$ but $z = 0$ then  
$$ \begin{pmatrix} wI_{p_1} & X \\ 0 & I_{p_2} \end{pmatrix} \cdot \bA = \begin{pmatrix} \bB_{1,1} & 0 \\ 0 & \bB_{2,2} \end{pmatrix} = \bA_{ss}.$$
Therefore $(\bA,\bA') \in G^2 \cdot \cc_{(p_1,p_2), (12)} \subseteq \overline{G^2 \cdot \cc_{(p), \id}}$, with the latter inclusion a consequence of  Lemma \ref{inclusions}, which shows (i). 

The case $z \neq 0$ and $w=0$ is obviously similar. Finally if $z=w=0$ then necessarily $p_1=p_2$ and $\bB_{1,1} = h \cdot \bB_{2,2}$ for some $h \in \GL_{p_1}$. Let \[g = \begin{pmatrix} I_{p_1} & 0 \\ 0 & h \end{pmatrix} \in G,\] then 
So $(g \cdot \bA, g \cdot \bA') \in \cc_{(p_1,p_2), (12)} \subseteq \overline{G^2 \cdot \cc_{(p), \id}}$ which shows (i).


\end{proof}

\begin{cor}
We have an inclusion $((p_1,p_2) \id) \preceq ((p), \id)$ if and only if $n=2$ and $p=2$.
\end{cor}

\begin{proof}
Choose arbitrary simple non-isomorphic $\bB_{1,1} \in \M^n_{p_1}$ and $\bB_{2,2} \in \M^n_{p_2}$. As these are not isomorphic, and simple, there is no nonzero homomorphism between these representations. Therefore $d_{\bB_{1,1}, \bB_{2,2}}$ is injective. The image of $d_{\bB_{1,1}, \bB_{2,2}}$ has dimension $p_1p_2$ and its codomain has dimension $np_1p_2$; provided $np_1p_2 \geq 2+p_1p_2$ there exists a pair of $n$-matrices $\bB_{1,2}, \bB'_{1,2} \in \M^n_{p_1,p_2}$ such that no linear combination of these lies in that image. The pair $(\bA,\bA') \in \cc_{(p_1,p_2), \id}$ with nonzero blocks given by $\bB_{i,j}, \bB'_{i,j}$ then does not lie in  $\overline{G^2 \cdot \cc_{(p), \id}}$ by Proposition \ref{newrkcondition}. Therefore  $((p_1,p_2) \id) \not \preceq ((p), \id)$.

Conversely assume that $p=2$. Then $p_1=p_2=1$ and any matrix in $\M^n_{p_1}$ is simple. Moreover, we may reformulate (iii) of Proposition  \ref{newrkcondition} as:
\begin{equation}\label{oldrkcondition} \rk(\ba_{11}-\ba_{22}, \ba_{12}, \ba'_{12}) < 3. \end{equation}
As this condition is vacuous when $n=2$ we get an inclusion $\cc_{(1,1),\id} \subseteq  \overline{G^2 \cdot \cc_{(2), \id}}$. Therefore  $((p_1,p_2), \id) \preceq ((p), \id)$ as required.
\end{proof}

Corollary \ref{reverses} showed that for all $k$ and all $\pi \in \Pi(p,k)$ we have an inclusion $(\pi, \sigma) \preceq ((p), \id)$ where $\sigma = [k,k-1, \ldots, 1]$. Consequently,

\begin{equation*} \bigcup_{k=1}^p \bigcup_{\pi \in \Pi(p,k)} G^2 \cdot \mathcal{C}_{\pi, [k,k-1, \ldots, 1]} \subseteq \overline{G^2 \cdot \mathcal{C}_{(p), \id}}. 
\end{equation*}

In fact it turns out that the above is an equality. We will need the following result which is somewhat reminiscent of the proof of the Hilbert-Mumford theorem.
In the statement and proof we write $T(R)$ for the subgroup of diagonal matrices in $\GL_p(R)$, for any ring $R$, and $\vv(R) = \vv \otimes_{\C} R$ for any ring $R$ containing $\C$.

\begin{Lemma}\label{iwahori}
Let $(\bA,\bA') \in \overline{\Gamma_{G,\vv}}$. Then there exist $h, h' \in G$ such that $(h \cdot \bA, h' \cdot \bA') \in \overline{\Gamma_{T,\vv}}$.
\end{Lemma}

\begin{proof} The hypotheses tell us that there exist $\bA(t) \in \vv(\C[[t]])$, $g(t) \in \GL_p(\C((t)))$ with 
\[\lim_{t \rightarrow 0} \bA(t) = \bA\] and
\[\lim_{t \rightarrow 0} g(t) \cdot \bA(t) = \bA'.\]
Since $\C[[t]]$ is principle ideal domain and $\C((t))$ its field of fractions, there exist $h(t), h'(t) \in \GL_p(\C[[t]])$ such that $h'(t) g(t) h(t)^{-1}:= \lambda(t) \in \GL_p(\C((t)))$ is diagonal. Let $\bA'(t) =  g(t) \cdot \bA(t) \in  \vv(\C[[t]])$, then
\[\bA'(t) = h'(t)^{-1} \lambda(t) h(t) \cdot \bA(t)\] for all $t \in \C^*$. Therefore
\[\lim_{t \rightarrow 0} h'(t) \cdot \bA'(t) = \lim_{t \rightarrow 0} \lambda(t) \cdot (h(t) \cdot \bA(t)).\]
This shows that $(h \cdot \bA, h' \cdot \bA') \in \overline{\Gamma_{T,\vv}}$, where $h  = \lim_{t \rightarrow 0} h(t)$ and $h' = \lim_{t \rightarrow 0} h'(t)$.
\end{proof}

\begin{rem}
In the above we used the following basic linear algebraic fact: if $R$ is a principle ideal domain and $Q(R)$ its field of fractions, then for all $g \in \GL_p(Q(R))$ there exist $h,h' \in \GL_p(R)$ such that $h'gh^{-1} \in T(Q(R))$. This is an easy special case of a more general (and more difficult to prove) result known as the Cartan-Iwahori-Matsumoto decomposition \cite[p. 52]{Mumford}:  let $R$ be a discrete valuation ring with field of fractions $Q(R)$. Suppose $G$ is a reductive linear algebraic group defined over $Q(R)$ and $T$ a maximal torus in $G$. Let $g \in G$, then there exist $h,h' \in G(R)$ such that $h'gh^{-1} \in T$. Thus, using this result one can generalise Lemma \ref{iwahori} above to the setting where $G$ is an arbitrary reductive linear algebraic group and $T$ is a maximal torus in $G$.
\end{rem}

\begin{prop}\label{decompCpid}
We have an equality 
\begin{equation*} \bigcup_{k=1}^p \bigcup_{\pi \in \Pi(p,k)} G^2 \cdot \mathcal{C}_{\pi, [k,k-1, \ldots, 1]} = \overline{G^2 \cdot \mathcal{C}_{(p), \id}}. 
\end{equation*}
\end{prop}

\begin{proof}
Only the inequality $\supseteq$ needs to be proven, since $\subseteq$ is proved in Corollary \ref{reverses}. Accordingly let $(\bA,\bA') \in  \overline{G^2 \cdot \mathcal{C}_{(p), \id}}$. Recall that $\overline{G^2 \cdot \mathcal{C}_{(p), \id}} = \overline{\Gamma_{G,\vv}}$. Therefore by Lemma \ref{iwahori} there exist $h,h' \in G$ such that $(h \cdot \bA, h' \cdot \bA') \in  \overline{\Gamma_{T,\vv}}$. 

Write $\bB = h\cdot \bA$ and $\bB'= h' \cdot \bA'$. There exist $\bB(t), \bB'(t) \in \vv(\C[[t]])$ and $g(t) \in T(\C((t)))$ such that 
\begin{align*}
\lim_{t \rightarrow 0} \bB(t) &= \bB;\\
\lim_{t \rightarrow 0} \bB'(t) &= \bB';\\
g(t) \cdot \bB(t) &= \bB'(t) 
\end{align*}
for all $t \in \C^*$. We may assume
\[g(t) = \bigoplus_{i=1}^p t^{e_i}\]
for some $(e_1,e_2, \ldots, e_p) \in \Z^p$. We must therefore have:
\[
\left\{ \begin{array}{lr} \bb_{ij} = \bb'_{ij} & \ \text{if $e_i=e_j$};\\
					\bb_{ij} = 0 &  \ \text{if $e_i<e_j$};\\
					\bb'_{ij} = 0 &  \ \text{if $e_i>e_j$}. 
\end{array} \right.
 \]
By altering $h$ and $h'$ if necessary, we may assume $e_1 \leq e_2 \leq \cdots \leq e_p$. If  the vector $(e_1, e_2, \ldots, e_p)$ contains exactly $k$ unique entries $(f_1,f_2, \ldots, f_k)$, we define a partition $\pi = (p_1, p_2, \ldots, p_k) \in \Pi(p,k)$ by setting $p_i = |\{j:e_j=f_i\}|$. With these assumptions we see that $\bB$ and $\bB'$ are $\pi$-block $n$-matrices, with equal diagonal blocks, such that $\bB$ is $\pi$-block upper triangular and $\bB'$ is $\pi$-block lower triangular. Let $\sigma = [k,k-1,k-2, \ldots, 1]$, then we have
\[(\bB, s_{\pi, \sigma} \cdot \bB') \in \mathcal{C}_{\pi, \sigma}.\]
Therefore
\[(\bA, \bA') \in G^2 \cdot  \mathcal{C}_{\pi, \sigma}\]
as required.
\end{proof}

\begin{cor}\label{boundaryofCp}
We have
\begin{equation}
\overline{G^2 \cdot \mathcal{C}_{(p), \id}} =  G^2 \cdot \mathcal{C}_{(p), \id} \cup \bigcup_{\pi \in \Pi(p,2)} \overline{G^2 \cdot \mathcal{C}_{\pi, (12)}}.
\end{equation}
\end{cor}

\begin{proof}
The inclusion $\supseteq$ follows immediately from Proposition \ref{decompCpid}, so we need only show $\subseteq$.
We first show that, for all $\pi \in \Pi(p,2)$ we have
\begin{equation}\label{clCp1p212}
\overline{G^2 \cdot \mathcal{C}_{\pi, (12)}} = \bigcup_{k=2}^p \bigcup_{\pi' \in \Pi(p,k), \pi' \leq \pi} G^2 \cdot \mathcal{C}_{\pi, [k,k-1,\ldots, 1]}.
\end{equation}
The proof of the $\supseteq$ part of this claim is similar to the proof of Corollary \ref{reverses} and left to the reader. For the inclusion $\subseteq$, let $(\bA,\bA' )\in \overline{G^2 \cdot \mathcal{C}_{\pi, (12)}}$. By Corollary \ref{speculation} there exist $g, g' \in G$ such that 
\[ (g \cdot \bA,g'\cdot \bA' )\in \overline{(P_{(p_1,p_2)} \times P_{(p_2,p_1)}) \cdot \mathcal{C}_{\pi, (12)}}.\]
Let $\bB_{i,j}, \bB'_{i,j}$ be the $(p_1,p_2)$ blocks of $g \cdot \bA$, resp. $(p_2,p_1)$ blocks of $g' \cdot \bA'$. By Proposition \ref{clCpisigma} we have
\[(\bB_{1,1}, \bB'_{2,2}) \in \overline{\GL^2_{p_1} \cdot \cc_{(p_1,\id)}}\]
and
 \[(\bB_{2,2}, \bB'_{1,1}) \in \overline{\GL^2_{p_2} \cdot \cc_{(p_2,\id)}}.\]
By Proposition \ref{decompCpid} we have
\[(\bB_{1,1}, \bB'_{2,2}) \in  \bigcup_{k=1}^{p_1} \bigcup_{\pi \in \Pi(p_1,k)} G^2 \cdot \mathcal{C}_{\pi, [k,k-1, \ldots, 1]}\]
and
\[(\bB_{2,2}, \bB'_{1,1}) \in  \bigcup_{k=1}^{p_2} \bigcup_{\pi \in \Pi(p_2,k)} G^2 \cdot \mathcal{C}_{\pi, [k,k-1, \ldots, 1]}.\]
Therefore $$(g \cdot \bA,g'\cdot \bA' )\in \bigcup_{k=2}^p \bigcup_{\pi' \in \Pi(p,k), \pi' \leq \pi} G^2 \cdot \mathcal{C}_{\pi, [k,k-1,\ldots, 1]} $$
from which Equation \eqref{clCp1p212} follows. We now note that for every $\pi' \in \Pi(p,k)$, $k \geq 3$, there exists some $\pi \in \Pi(p,2)$ with $\pi' \leq \pi$. So the desired result now follows from the claim above and Proposition \ref{decompCpid}.
\end{proof}

As an incidental result, we have:

\begin{cor}
The orbit space $\overline{G^2 \cdot \cc_{\pi, \sigma}}$ is closed if and only if $|\pi|=p$.
\end{cor}

\begin{proof}
The if direction was established in the second part of Corollary \ref{speculation}, so we need only prove the only if direction. If $|\pi|=1$ then this is a consequence of Proposition \ref{decompCpid}. So suppose $1<|\pi|=k<p$. Then there exists $1 \leq i \leq k$ such that $p_i>1$. By the $|\pi|=1$ case, we may choose a pair $(\bB, \bB') \in  \overline{\GL^2_{p_i} \cdot \cc_{((p_i),\id)}} \setminus  \GL^2_{p_i} \cdot \cc_{((p_1),\id)}$. Construct a pair $(\bA,\bA') \in \overline{(P_{\pi} \times P_{\sigma(\pi)}) \cdot \cc_{\pi, \sigma}}$ by setting $\bB_{i,i}= \bB, \bB'_{\sigma(i), \sigma(i)} = \bB'$, choosing arbitrary simple $\bB_{j,j} = \bB'_{\sigma(j), \sigma(j)}$ for each $j \neq i$, and choosing $\bB_{i,j}, \bB'_{i,j}$ for $i>j$ arbitrarily. Then 
$(\bA,\bA') \not \in (P_{\pi} \times P_{\sigma(\pi)}) \cdot \cc_{\pi, \sigma}$, and therefore
$$ (P_{\pi} \times P_{\sigma(\pi)}) \cdot \cc_{\pi, \sigma} \subsetneq \overline{(P_{\pi} \times P_{\sigma(\pi)}) \cdot \cc_{\pi, \sigma}}.$$
The desired result now follows from the first part of Corollary \ref{speculation}.
\end{proof}

We need one further new concept: let $(\pi, \sigma) \in \mathcal{P}_{p,n}$ and let $(\bA,\bA') \in \mathcal{C}_{\pi, \sigma}$. Suppose for the time being that $n \geq 3$. We shall say the pair $(\bA,\bA')$ is {\it supermaximally general} if the following hold:
\begin{enumerate}
\item $\bA$ is maximally general for $\pi$;
\item $\bA'$ is maximally general for $\sigma(\pi)$;
\item for all $l=1,\ldots, k$ satsifying $\sigma(l+1) = \sigma(l)+1$, we have that the sub-$\pi$-block $n$-matrix pair
\[\left( \begin{pmatrix}\bB_{l,l} & \bB_{l,l+1} \\ 0 & \bB_{l+1,l+1} \end{pmatrix}, \begin{pmatrix} \bB_{\sigma(l),\sigma(l)} & \bB_{\sigma(l),\sigma(l)+1} \\ 0 & \bB_{\sigma(l)+1,\sigma(l)+1} \end{pmatrix} \right) \not \in \overline{\GL^2_{p_l+p_{l+1}} \cdot \mathcal{C}_{(p_l+p_{l+1}), \id}}.\]
\end{enumerate}

If $n=2$ we weaken the definition by ignoring condition (3) in cases where $p_{l}=p_{l+1}=1$.  It is a consequence of Proposition \ref{newrkcondition} that for every pair $(\pi, \sigma)$, the space $\mathcal{C}_{\pi, \sigma}$ contains a supermaximally general pair $(\bA,\bA')$ with an arbitrary choice simple diagonal $\pi$-blocks in $\bA$.

\begin{Lemma}\label{supermax} Let $(\bA,\bA') \in \cc_{\pi, \sigma}$ be supermaximally general and suppose that $(\bA,\bA') \in \overline{G \cdot \cc_{(p), \id}}$. Then there is a chain of inclusions
\[(\pi, \sigma) = (\pi_0, \sigma_0) \preceq (\pi_1,\sigma_1) \preceq \cdots \preceq (\pi_{k-1}, \sigma_{k - 1}) = ((p), \id).\]
\end{Lemma}

\begin{proof} The proof is by induction on $k = |\pi|$. When $k=2$ the result follows immediately from Lemma \ref{newrkcondition}, so we assume $k>2$.
It is clear that $(\bA,\bA') \not \in G^2 \cdot \mathcal{C}_{(p), \id}$, so by Corollary \ref{boundaryofCp} we must have $(\bA,\bA')  \in \overline{G^2 \cdot \mathcal{C}_{(p_1,p_2), (12)}}$ for some $(p_1,p_2) \in \Pi(p,2)$. As both $\bA$ and $\bA'$ are maximally general, $(\pi, \sigma)$ must be compatible with $((p_1,p_2),(12))$. This means that there exist $\kappa_1 \in \Pi(p_1,k_1)$, $\kappa_2 \in \Pi(p_2,k_2)$, $\tau_1 \in S_{k_1}$, $\tau_2 \in S_{k_2}$ such that
\[(\bB_{1,1}, \bB'_{2,2}) \in \overline{\GL^2_{p_1} \cdot \cc_{(\kappa_1, \tau_1)}}, (\bB_{2,2}, \bB'_{1,1}) \in \overline{\GL^2_{p_2} \cdot \cc_{(\kappa_2, \tau_2)}},\] where $\kappa_1$ and $\kappa_2$ are the partitions of $p_1$ and $p_2$ into $k_1$ and $k_2$ blocks respectively induced by $\pi$. Since $(\bA, \bA') \in  \overline{G^2 \cdot \mathcal{C}_{(p_1,p_2), (12)}}$ we also have 
\[(\bB_{1,1}, \bB'_{2,2}) \in \overline{\GL^2_{p_1} \cdot \cc_{((p_1), \id)}}, (\bB_{2,2}, \bB'_{1,1}) \in \overline{\GL^2_{p_2} \cdot \cc_{((p_2), \id)}}.\] Further, $(\bB_{1,1}, \bB'_{2,2})$ is supermaximally general, as is $(\bB_{2,2}, \bB'_{1,1})$. Therefore by induction there are chains
 \[(\kappa_i, \tau_i) = (\kappa_{i,0}, \tau_{i,0}) \preceq (\kappa_{i,1},\tau_{i,1}) \preceq \cdots \preceq (\kappa_{k_i-1}, \tau_{i,k_i - 1}) = ((p_i), \id)\] for $i=1,2$.
Define, for each $i=0, \ldots, k_1-1$
\[\pi_i = (\kappa_{1,i}, \kappa_2)\]
and for each $i=1, \ldots, k_2-1$
\[\pi_{i+k_1-1} = (p_1,\kappa_{2,i}).\]
Further for each $i=0, \ldots, k_1-1$ set
\[\sigma_i = (\tau_2,\tau_{1,i})^*\]
and for each $i=1, \ldots, k_2-1$
\[\sigma_{i+k_1-1} = (\tau_{i,2},\id)^*.\]
Then we have a chain 
\[(\pi, \sigma) = (\pi_0, \sigma_0) \preceq  (\pi_1,\sigma_1) \preceq \cdots \preceq (\pi_{k-2}, \sigma_{k - 2}) = ((p_1,p_2), (12)) \preceq  ((p), \id)\] as required.
\end{proof}

\begin{rem}\label{supermaxcor}
As a consequence of the above, we obtain Theorem \ref{saturated} in the special case where $(\hat{\pi}, \hat{\sigma}) = ((p), \id)$.
\end{rem}

\begin{prop}\label{compatconverse} Let $(\pi, \sigma), (\hat{\pi}, \hat{\sigma}) \in \mathcal{P}_{p,n}$. Then $(\pi, \sigma) \preceq (\hat{\pi}, \hat{\sigma})$ if and only if the pairs are compatible and, in the notation of Proposition \ref{compatibility} we have $$(\kappa_i, \tau_i) \preceq ((\hat{p}_i), \id)$$ for all $i=1, \ldots, \hat{k}$.
\end{prop}

\begin{proof} Only the necessity of the displayed equation needs to be shown. Suppose for some $i$ it does not hold. Choose $(\bA,\bA') \in \cc_{\pi, \sigma}$ with arbitrary simple diagonal $\pi$-blocks $\bB_{j,j} = \bB'_{\sigma(j),\sigma(j)}$ for $j=1, \ldots, k$. We may choose superdiagonal blocks $\bB_{j,j+1}$ and $\bB'_{j,j+1}$ for $j=1, \ldots, k-1$ in such a way that $(\hat{\bB}_{i,i}, \hat{\bB'}_{\hat{\sigma}(i), \hat{\sigma}(i)})$ is supermaximally general, and $\bA$ and $\bA'$ are maximally general. By Lemma \ref{supermax} we have $(\hat{\bB}_{i,i},\hat{\bB'}_{\hat{\sigma}(i), \hat{\sigma}(i)}) \in \cc_{\kappa_i, \tau_i} \setminus \overline{\GL_{p_i}^2 \cdot \cc_{p_i, \id}}$. By Lemma \ref{clCpisigma} we find that $(\bA,\bA') \not \in \overline{G^2 \cdot \cc_{\hat{\pi}, \hat{\sigma}}}$, so that  $(\pi, \sigma) \not\preceq (\hat{\pi}, \hat{\sigma})$.
\end{proof}

\begin{proof}[Proof of Theorem \ref{saturated}]
The if direction is clear. Suppose $(\pi, \sigma) \preceq (\hat{\pi}, \hat{\sigma})$. Then by Proposition \ref{compatconverse} the pairs are compatible and, in the notation of Proposition \ref{compatibility} we have $$(\kappa_i, \tau_i) \preceq ((\hat{p}_i), \id)$$ for all $i=1, \ldots, \hat{k}$. By Remark \ref{supermaxcor} we have chains  
 \[(\kappa_i, \tau_i) = (\kappa_{i,0}, \tau_{i,0}) \preceq (\kappa_{i,1},\tau_{i,1}) \preceq \cdots \preceq (\kappa_{k_i-1}, \tau_{i,k_i - 1}) = ((p_i), \id)\] for $i=1, \ldots, \hat{k}$.
Define, for each $i=0, \ldots, k_1-1$
\[\pi_i = (\kappa_{1,i}, \kappa_2, \ldots, \kappa_{\hat{k}})\]
and for each $j=1, \ldots, \hat{k}$ and each $i=1, \ldots, k_j-1$
\[\pi_{i+k_1+\cdots+k_{j-1}-1} = (p_1,p_2, \ldots, \kappa_{j,i},\kappa_{j+1}, \ldots, \kappa_{\hat{k}}).\]
Further for each $i=0, \ldots, k_1-1$ set
\[\sigma_i = ((\tau_{1,i}, \tau_2, \ldots, \tau_{\hat{k}})^{\hat{\sigma}})^*\]
and for each $j=1, \ldots, \hat{k}$ and each $i=1, \ldots, k_j-1$
\[\sigma_{i+k_1+\cdots+k_{j-1}-1} = ((\id,\id,\ldots, \tau_{j,i}, \tau_{j+1}, \ldots, \tau_{\hat{k}})^{\hat{\sigma}})^*.\]
Then we have a chain
\[(\pi, \sigma) = (\pi_0, \sigma_0) \preceq  (\pi_1,\sigma_1) \preceq \cdots \preceq (\pi_{k-1}, \sigma_{k - 1}) = (\hat{\pi}, \hat{\sigma})\] as required.
\end{proof}

\section{Examples: $n \geq 3$}
We want to find the irreducible components of $\mathcal{S}_{G,\vv}$. As shown in the previous, section, these are the orbit spaces $\overline{G^2 \cdot \cc_{\pi, \sigma}}$ where $(\pi, \sigma) \in \mathcal{P}_{p,n}$ is maximal. It turns out that if $n \geq 3$ then the maximal elements are easy to describe - indeed, maximality depends only on the permutation $\sigma$. 

We will say $\sigma \in S_k$ is a {\it partial reversal} if $\sigma$ in square bracket notation contains a substring $[l+1,l]$, i.e. if for some $1  \leq l< k$ we have \begin{equation}\label{partialreversal} \sigma(l+1)+1 = \sigma(l).\end{equation}

\begin{prop}\label{nonmax} Let $n \geq 3$ and let $(\pi, \sigma) \in \mathcal{P}_{p,n}$. Then $(\pi, \sigma)$ is maximal if and only if $\sigma$ is not a partial reversal.
\end{prop}

\begin{proof}
Suppose $\sigma \in S_k$ is a partial reversal with $l$ satisfying \eqref{partialreversal}. Define
\[\hat{\pi} = (p_1, p_2, \ldots, p_{l-1}, p_l+p_{l+1}, p_{l+2}, \ldots, p_k) \in \Pi(p,k-1)\]
and
\[\hat{\sigma} = [\sigma(1), \sigma(2), \ldots, \sigma(l-1), \sigma(l+1), \ldots, \sigma(k)] \in S_{k-1}.\]
Then $\pi \leq \hat{\pi}$ and $$\sigma = ((\id, \id, \ldots, \id, (12), \id, \ldots, \id)^{\hat{\sigma}})^*,$$ and $((p_l,p_{l+1}), (12)) \preceq ((p_l+p_{l+1}), \id)$. By Proposition \ref{compatibility} we have $(\pi, \sigma) \preceq (\hat{\pi}, \hat{\sigma})$ and so $(\pi, \sigma)$ is a not maximal. 

Conversely, suppose we have an inclusion $(\pi, \sigma) \preceq (\hat{\pi}, \hat{\sigma})$ for some $(\hat{\pi}, \hat{\sigma}) \in \mathcal{P}_{p,n}$. 
By Theorem \ref{saturated} we may assume that $\hat{k} = k-1$. Then by Proposition \ref{compatibility} we have $\sigma = ((\id, \id, \ldots, \id, \kappa, \id, \ldots, \id)^{\hat{\sigma}})^*$, where the $\kappa$ appears at index $l$, and $((p_l,p_{l+1}),\kappa) \preceq (p_l+p_{l+1}, \id)$. Since $n \geq 3$, by Proposition \ref{newrkcondition} we must have $\kappa=(12)$. Then $\sigma$ in square bracket notation contains the substring $[l+1,l]$. 
\end{proof}

\begin{Theorem}\label{sepvardecomp} For all $p \geq 2$ and $n \geq 3$ we have a complete irredundant decomposition 
\begin{equation}\label{betterdecomp} \mathcal{S}_{G,\vv} = \bigcup_{k=1}^p \bigcup_{\pi \in \Pi(p,k)} \bigcup_{\sigma \in T_k} \overline{G^2 \cdot \cc_{\pi,\sigma}},\end{equation}
where $T_k$ is the subset of $S_k$ consisting of permutations which are not partial reversals.
\end{Theorem}

The problem of enumerating $|T_k|$ is a classical one in enumerative combinatorics. One may show that it satisfies the recurrence
\begin{equation}\label{TkRecurrence} |T_k| = (k-1)|T_{k-1}| + (k-2)|T_{k-2}|
\end{equation} 
for $k \geq 3$, and obviously $|T_1|=|T_2|=1$. To see this, suppose $\sigma \in T_k$ is given in square bracket form, and consider the permutation in $S_{k-1}$ obtained by removing the entry $k$ from the list. The resulting permutation is either in $T_{k-1}$ or it is not. For each element of $T_{k-1}$ we can construct $(k-1)$ different elements of $T_k$ by inserting the entry $k$ anywhere except immediately prior to $k-1$. On the other hand, if when we remove $k$ we get an element of $S_{k-1} \setminus T_{k-1}$, that element must have exactly one substring of the form $[l+1,l]$. The entry immediately after this substring cannot be $l-1$, thus if we remove $l$ and decrease each entry $>l$ in the string by 1 we obtain an element of $T_{k-2}$. Given an element of $T_{k-2}$, we may construct an element of $S_{k-1}$ by replacing any of its $k-2$ entries $l$ by the string $[l+1,l]$ and increasing each entry $>l$ by 1, and for each such element we obtain a unique element of $T_k$ by inserting $k$ between $l+1$ and $l$.

One may prove by induction that for $k \geq 2$ we have
\begin{equation}
|T_k| = \sum_{q=1}^{k-1} \binom{k-1}{q}(-1)^q(k-q)!.
\end{equation}   

The first few values of $|T_k|$ are $1,1,3,11,53,309$, and it is A000255 in the OEIS \cite{OEIS}.

For $n \geq 3$ the number of components of $\mathcal{S}_{G,\vv}$ of dimension $(n+1)p^2-k$ is given by
\[\binom{p-1}{k-1}|T_k|\] with binomial term here being $|\Pi(p,k)|$. In particular, $T_p$ is the number of components of $\mathcal{S}_{G, \vv}$ with (minimum possible) dimension $(n+1)p^2-p$. Since it is easy to see that $|T_p| \geq 1$ for every $p$ (for example, the identity is in $T_p$ for every $p$), we get that $\sdim(\mathcal{S}_{G,\vv}) = (n+1)p^2-p$. This completes the proof of Theorem \ref{mainsepvar} in the case $n \geq 3$.

We conclude this section with some explicit examples of the posets $\mathcal{P}_{p,n}$ for $n \geq 3$ and the associated decompositions. To save space we use noncommutative multiplicative notation for partitions, so for instance $1^3 \cdot 2 \cdot 1$ means $(1,1,1,2,1)$.

\begin{eg}\label{decomp2}
The poset $\mathcal{P}_{2,n}$ is given in Figure \ref{decomp2figure}. 
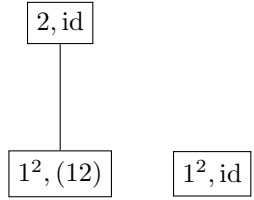
\begin{figure}[ht]\caption{The poset $\mathcal{P}_{2,n}$, $n \geq 2$.}\label{decomp2figure}
\begin{center}
\begin{tikzpicture}[every node/.style={draw}]
  \path[yshift=1.5cm,shape=rectangle];
\node(A) at (0,2){$2, \id$};
\node(B) at (0,0){$1^2, (12)$};
\node(C) at (2,0){$1^2, \id$};
\draw (A)->(B);
\end{tikzpicture}
\end{center}
\end{figure}

Note that the identity in $S_k$ is not a partial reversal, for all $k$. $(12) \in S_2$ is a partial reversal. Therefore we have an irredundant decomposition
 \[\mathcal{S}_{G,\vv} =  \overline{G^2 \cdot \cc_{(2),\id}} \cup G^2 \cdot \cc_{(1,1),\id}\] for $p=2$ and any $n \geq 3$.
The dimensions of these components are $4n+3$ and $4n+2$ respectively. This was first shown in \cite{ElmerMatrixSepVar}, where it was also shown that the intersection of these components has dimension $3n+4$.
\end{eg}

\begin{eg}\label{decomp3}
The poset $\mathcal{P}_{3,n}$ is given in Figure \ref{decomp3figure}.

\begin{figure}[ht]\caption{The poset $\mathcal{P}_{3,n}$, $n \geq 3$.}\label{decomp3figure}
\begin{center}
\begin{tikzpicture}[every node/.style={draw},scale=0.8]
  \path[yshift=1.5cm,shape=rectangle];
\node(A) at (0,2){$2 \cdot 1, \id$};
\node(B) at (0,0){$1^3, (12)$};
\node(C) at (2.5,0){$1^3, \id$};
\node(D) at (5,2){$1 \cdot 2, \id$};
\node(E) at (5,0){$1^3, (23)$};
\node(F) at (-5,4){$3, \id$};
\node(G) at (-3.5,2){$2 \cdot 1,(12)$};
\node(H) at (-6.5,2){$1 \cdot 2,(12)$};
\node(I) at (-5,0){$1^3,(13)$};
\node(J) at (-2.5,0){$1^3,(123)$};
\node(K) at (-7.5,0){$1^3,(132)$};
\draw (A)->(B);
\draw (D)->(E);
\draw(F)->(G);
\draw(F)->(H);
\draw(G)->(I);
\draw(H)->(I);
\end{tikzpicture}
\end{center}
\end{figure}
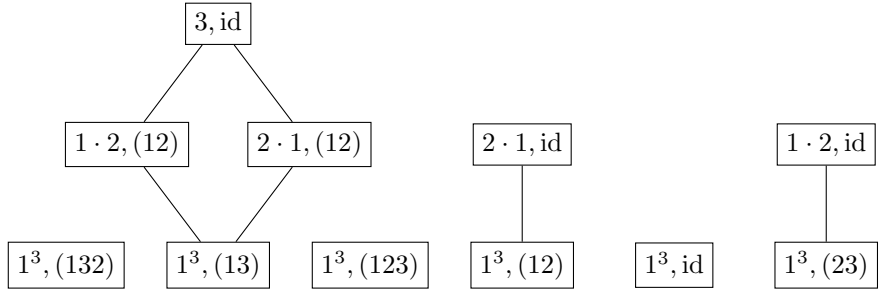
We have an irredundant decomposition
\begin{align*}
\mathcal{S}_{G,\vv} &= \overline{G^2 \cdot \cc_{(3),\id}} \cup \overline{G^2 \cdot \cc_{(2,1),\id}} \cup 
\overline{G^2 \cdot \cc_{(1,2),\id}} \\ &\cup {G^2 \cdot \cc_{(1,1,1),\id}} 
\cup {G^2 \cdot \cc_{(1,1,1),(123)}} \cup {G^2 \cdot \cc_{(1,1,1),(132)}}
\end{align*}
for $p=3$ and any $n \geq 3$. The first component here has dimension $9n+8$, the next two $9n+7$ and the remaining three dimension $9n+6$.
\end{eg}

\begin{eg}\label{decomp4}
The structure of the poset $\mathcal{P}_{4,n}$ for any $n \geq 3$ is given in Figure \ref{poset4} overleaf. $\mathcal{S}_{G,\vv}$ has an irredundant decomposition with one component of dimension $16n+15$, three of dimension $16n+14$, nine  of dimension $16n+13$ and eleven of dimension $16n+12$.  
\end{eg}

\begin{landscape}
\thispagestyle{empty}
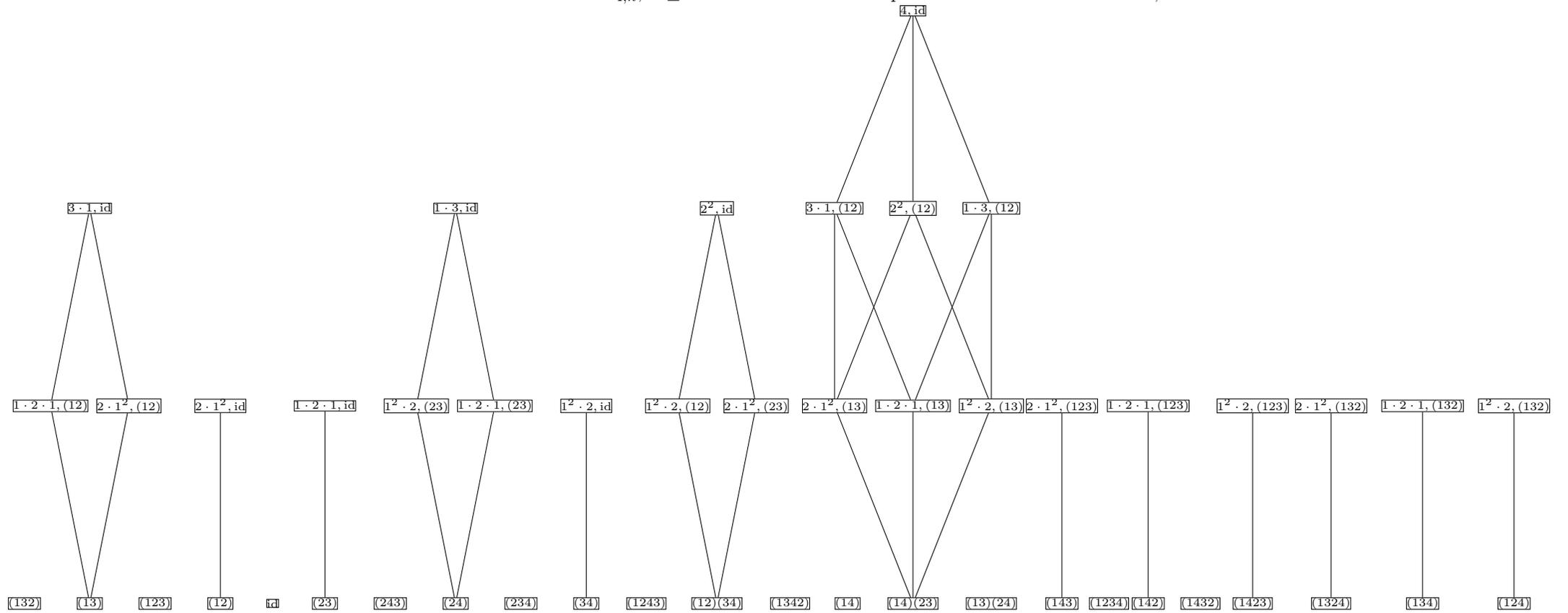
\begin{figure}
\centering 
\kern5.5pt

\caption{The Poset $\mathcal{P}_{4,n}$, $n \geq 3$. We have omitted the partitions from the bottom row, which are all $1^4$.}\label{poset4}
\begin{tiny}
\noindent\hspace{-5.2cm}\begin{tikzpicture}[every node/.style={inner sep=0pt, draw},xscale=0.48,yscale=1.8]
  \path[shape=rectangle];
\node(A) at (0,2){$2 \cdot 1^2, \id$};
\node(B) at (0,0){$(12)$};
\node(C) at (2,0){$\id$};
\node(D) at (4,2){$1 \cdot 2 \cdot 1, \id$};
\node(E) at (4,0){$(23)$};
\node(F) at (-5,4){$3\cdot 1, \id$};
\node(G) at (-3.5,2){$2 \cdot 1^2,(12)$};
\node(H) at (-6.5,2){$1 \cdot 2\cdot 1,(12)$};
\node(I) at (-5,0){$(13)$};
\node(J) at (-2.5,0){$(123)$};
\node(K) at (-7.5,0){$(132)$};
\draw (A)->(B);
\draw (D)->(E);
\draw(F)->(G);
\draw(F)->(H);
\draw(G)->(I);
\draw(H)->(I);
\node(O) at (14,2){$1^2 \cdot 2, \id$};
\node(P) at (14,0){$ (34)$};
\node(Q) at (9,4){$1 \cdot 3, \id$};
\node(R) at (10.5,2){$1 \cdot 2 \cdot 1,(23)$};
\node(S) at (7.5,2){$1^2 \cdot 2,(23)$};
\node(T) at (9,0){$(24)$};
\node(U) at (11.5,0){$(234)$};
\node(V) at (6.5,0){$(243)$};
\draw (O)->(P);
\draw(Q)->(R);
\draw(Q)->(S);
\draw(R)->(T);
\draw(S)->(T);
\node(U) at (19,4){$2^2, \id$};
\node(V) at (20.5,2){$2 \cdot 1^2,(23)$};
\node(W) at (17.5,2){$1^2 \cdot 2,(12)$};
\node(X) at (19,0){$(12)(34)$};
\draw(U)->(V);
\draw(U)->(W);
\draw(V)->(X);
\draw(W)->(X);
\node(Y) at (26.5,6){$4, \id$};
\node(Z) at (26.5,4){$2^2, (12)$};
\node(AA) at (23.5,4){$3 \cdot 1, (12)$};
\node(AB) at (29.5,4){$1 \cdot 3, (12)$};
\node(AC) at (23.5,2){$2 \cdot 1^2, (13)$};
\node(AD) at (29.5,2){$1^2 \cdot 2, (13)$};
\node(AE) at (26.5,0){$(14)(23)$};
\node(AF) at (26.5,2){$1 \cdot 2 \cdot 1, (13)$};
\draw (Y)->(Z);
\draw(Y)->(AA);
\draw(Y)->(AB);
\draw(AA)->(AC);
\draw(Z)->(AC);
\draw(AB)->(AD);
\draw(Z)->(AD);
\draw(AC)->(AE);
\draw(AD)->(AE);
\draw(AF)->(AE);
\draw(AA)->(AF);
\draw(AB)->(AF);
\node(AG) at (32.2,2){$2 \cdot 1^2, (123)$};
\node(AH) at (32.2,0){$(143)$};
\draw (AG)->(AH);
\node(AI) at (35.5,2){$1 \cdot 2 \cdot  1, (123)$};
\node(AJ) at (35.5,0){$(142)$};
\draw (AI)->(AJ);
\node(AI) at (39.5,2){$1^2 \cdot 2, (123)$};
\node(AJ) at (39.5,0){$(1423)$};
\draw (AI)->(AJ);
\node(AK) at (42.5,2){$2 \cdot 1^2, (132)$};
\node(AL) at (42.5,0){$(1324)$};
\draw (AK)->(AL);
\node(AM) at (46,2){$1 \cdot 2 \cdot  1, (132)$};
\node(AN) at (46,0){$(134)$};
\draw (AM)->(AN);
\node(AO) at (49.5,2){$1^2 \cdot 2, (132)$};
\node(AP) at (49.5,0){$(124)$};
\draw (AO)->(AP);

\node(AQ) at (24,0){$(14)$};
\node(AR) at (34,0){$(1234)$};
\node(AS) at (37.5,0){$(1432)$};
\node(AT) at (29.5,0){$(13)(24)$};
\node(AU) at (21.8,0){$(1342)$};
\node(AV) at (16.3,0){$(1243)$};
\end{tikzpicture}
\end{tiny}
\end{figure}
\end{landscape}

\section{Examples: $n=2$}

In this section we fix $n=2$. The problem of finding the maximal elements of $\mathcal{P}_{p,n}$ in this case does not seem to have a simple answer. In general, the question of whether $(\pi, \sigma)$ is maximal depends on both $\pi$ and $\sigma$. The analogue of Proposition \ref{nonmax}, proved in almost exactly the same way, is as follows:

\begin{prop} Let $(\pi, \sigma) \in \mathcal{P}_{p,2}$. Then $(\pi, \sigma)$ is maximal if and only if $\sigma$ in square bracket notation does not contain a substring $[l+1,l]$ for any $l$, and in addition does not contain a substring $[l,l+1]$ for any $l$ for which $p_l=p_{l+1}=1$.
\end{prop} 

We were unable to find a formula for the number of maximal elements of arbitrary rank in $\mathcal{P}_{p,2}$. However, we are most interested in determining the subdimension of $\mathcal{S}_{G, \vv}$, and we can extract enough information to do this.

The above proposition implies in particular, that if $|\pi|=p$, so that $p_i=1$ for all $i=1, \ldots,p$, then $(\sigma, \pi)$ is maximal if and only if $\sigma$ in square bracket notation contains neither a substring $[l,l+1]$ nor a substring $[l+1,l]$. It is easily seen that for $p \geq 4$, at least one such permutation exists, e.g. $[2,4, \ldots, p-1, 1,3, \ldots, p]$ if $p$ is odd and $[2,4, \ldots, p,1,3, \ldots, p-1]$ if $p$ is even. Therefore, for all $p \geq 4$,  $\mathcal{P}_{p,2}$ contains a maximal element of rank $p$. It follows that $\sdim(\mathcal{S}_{G,\vv}) = (n+1)p^2-p$ as required.  
Together with Proposition \ref{sepvardecomp}, this completes the proof of Theorem \ref{mainsepvar}. 

Let $U_p$ be the set of permutations in $S_p$ containing neither a substring $[l,l+1]$ nor $[l+1,l]$. Determining $u_p:=|U_p|$ is another classical problem in enumerative combinatorics, known as {\it Hertzsprung's problem}. The problem is usually expressed less formally in terms of a chessboard: in how many ways can one place $p$ kings on a $p \times p$ chessboard, with one in each row and column, such that no two are diagonally adjacent?
A closed formula for $u_p$ was given by Hertzsprung (see \cite{Claesson}); later Abramson and Moser \cite{AbramsonMoser} showed that
\[u_p = \sum_{k=0}^{p-1} (-1)^k (p-k)! \sum_{i=0}^k \binom{p-k}{i} \binom{p-1-i}{k-i}.\]

The first few terms of the sequence $u_p$ are $1,0,0,2,14,90$, and it is sequence A002464 in the OEIS \cite{OEIS}. The number $u_p$ is the exact number of components of $\mathcal{S}_{G,\M^2_p}$ with codimension $p-1$.

\begin{eg}\label{decomp22}
The poset $\mathcal{P}_{2,2}$ has the following structure:

\begin{figure}[h]\caption{The poset $\mathcal{P}_{2,2}$.}
\begin{center}
\begin{tikzpicture}[every node/.style={draw}]
  \path[yshift=1.5cm,shape=rectangle];
\node(A) at (0,2){$2, \id$};
\node(B) at (0,0){$1^2, (12)$};
\node(C) at (2,0){$1^2, \id$};
\draw (A)->(B);
\draw (A) -> (C);
\end{tikzpicture}
\end{center}
\end{figure}
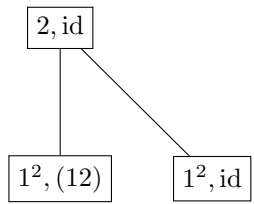
It follows that $\mathcal{S}_{G,\vv} = \overline{G^2 \cdot \cc_{(2), \id}}$ is irreducible when $p=2$ and $n=2$. This was proved earlier in \cite{ElmerMatrixSepVar}.
\end{eg}

\newpage
\begin{eg}\label{decomp32}
The poset $\mathcal{P}_{3,2}$ has the following structure:

\begin{figure}[h]\caption{The poset $\mathcal{P}_{3,2}$.}\label{poset32}
\begin{center}
\begin{tikzpicture}[every node/.style={draw},scale=0.8]
  \path[yshift=1.5cm,shape=rectangle];
\node(A) at (0,2){$2 \cdot 1, \id$};
\node(B) at (0,0){$1^3, (12)$};
\node(C) at (2.5,0){$1^3, \id$};
\node(D) at (5,2){$1 \cdot 2, \id$};
\node(E) at (5,0){$1^3, (23)$};
\node(F) at (-5,4){$3, \id$};
\node(G) at (-3.5,2){$2 \cdot 1,(12)$};
\node(H) at (-6.5,2){$1 \cdot 2,(12)$};
\node(I) at (-5,0){$1^3,(13)$};
\node(J) at (-2.5,0){$1^3,(123)$};
\node(K) at (-7.5,0){$1^3,(132)$};
\draw (A)->(B);
\draw (G) -> (J);
\draw (H) -> (K);
\draw (D)->(E);
\draw(F)->(G);
\draw(F)->(H);
\draw(G)->(I);
\draw(H)->(I);
\draw (A) -> (C);
\draw(C)-> (D);
\end{tikzpicture}
\end{center}
\end{figure}
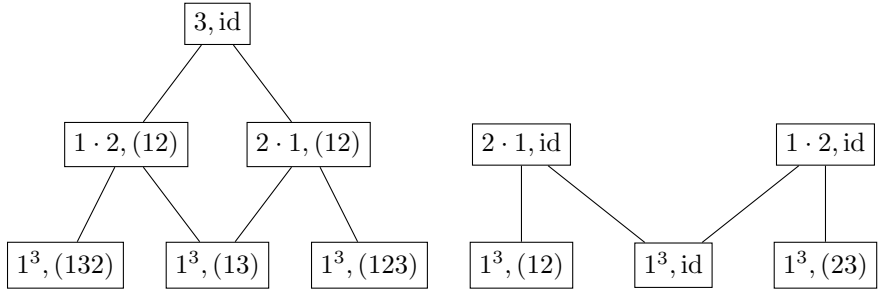
We have an irredundant decomposition
\begin{align*}
\mathcal{S}_{G,\vv} &= \overline{G^2 \cdot \cc_{(3),\id}} \cup \overline{G^2 \cdot \cc_{(2,1),\id}} \cup 
\overline{G^2 \cdot \cc_{(1,2),\id}} 
\end{align*}
for $p=3$ and $n=2$. The first component here has dimension $26$, the other two have dimension $25$. In particular there are no components with codimension 2. This serves as a sanity check - if there were components of codimension 2 then we could conclude $\C[\M_3^2]^G$ contains no polynomial or hypersurface separating set. However, as we mentioned in the introduction, $\C[\M_3^2]^G$ is a hypersurface.
\end{eg}

\begin{eg} The poset $\mathcal{P}_{4,2}$ is given in Figure \ref{poset42} overleaf. It is quite complicated, but note the two maximal elements of rank $4$, namely $(1^4, (1243))$ and $(1^4,(1342))$. The associated irredundant decomposition of $\mathcal{S}_{G,\vv}$ has one component of dimension $47$, three of dimension $46$, five of dimension $45$ and two of dimension $44$.
\end{eg}

\begin{landscape}
\thispagestyle{empty}
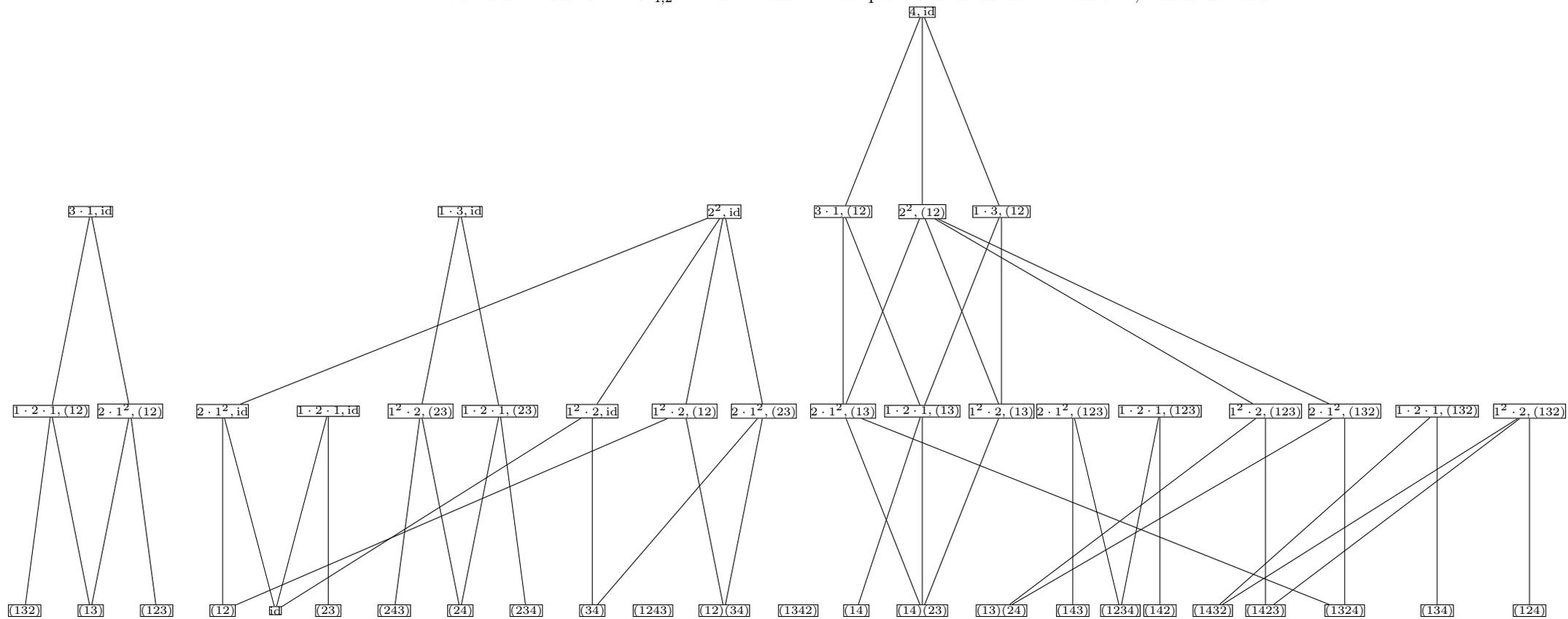
\begin{figure}
\centering 
\kern5.5pt

\caption{The Poset $\mathcal{P}_{4,2}.$ We have omitted the partitions from the bottom row, which are all $1^4$.}\label{poset42}
\begin{tiny}
\noindent\hspace{-5.2cm}\begin{tikzpicture}[every node/.style={inner sep=0pt, draw},xscale=0.48,yscale=1.8]
  \path[shape=rectangle];
\node(A) at (0,2){$2 \cdot 1^2, \id$};
\node(B) at (0,0){$(12)$};
\node(C) at (2,0){$\id$};
\node(D) at (4,2){$1 \cdot 2 \cdot 1, \id$};
\node(E) at (4,0){$(23)$};
\node(F) at (-5,4){$3\cdot 1, \id$};
\node(G) at (-3.5,2){$2 \cdot 1^2,(12)$};
\node(H) at (-6.5,2){$1 \cdot 2\cdot 1,(12)$};
\node(I) at (-5,0){$(13)$};
\node(J) at (-2.5,0){$(123)$};
\node(K) at (-7.5,0){$(132)$};
\draw (A)->(B);
\draw (D)->(E);
\draw(F)->(G);
\draw(F)->(H);
\draw(G)->(I);
\draw(H)->(I);
\node(O) at (14,2){$1^2 \cdot 2, \id$};
\node(P) at (14,0){$ (34)$};
\node(Q) at (9,4){$1 \cdot 3, \id$};
\node(R) at (10.5,2){$1 \cdot 2 \cdot 1,(23)$};
\node(S) at (7.5,2){$1^2 \cdot 2,(23)$};
\node(T) at (9,0){$(24)$};
\node(L) at (11.5,0){$(234)$};
\node(M) at (6.5,0){$(243)$};
\draw (O)->(P);
\draw(Q)->(R);
\draw(Q)->(S);
\draw(R)->(T);
\draw(S)->(T);
\node(U) at (19,4){$2^2, \id$};
\node(V) at (20.5,2){$2 \cdot 1^2,(23)$};
\node(W) at (17.5,2){$1^2 \cdot 2,(12)$};
\node(X) at (19,0){$(12)(34)$};
\draw(U)->(V);
\draw(U)->(W);
\draw(V)->(X);
\draw(W)->(X);
\node(Y) at (26.5,6){$4, \id$};
\node(Z) at (26.5,4){$2^2, (12)$};
\node(AA) at (23.5,4){$3 \cdot 1, (12)$};
\node(AB) at (29.5,4){$1 \cdot 3, (12)$};
\node(AC) at (23.5,2){$2 \cdot 1^2, (13)$};
\node(AD) at (29.5,2){$1^2 \cdot 2, (13)$};
\node(AE) at (26.5,0){$(14)(23)$};
\node(AF) at (26.5,2){$1 \cdot 2 \cdot 1, (13)$};
\draw (Y)->(Z);
\draw(Y)->(AA);
\draw(Y)->(AB);
\draw(AA)->(AC);
\draw(Z)->(AC);
\draw(AB)->(AD);
\draw(Z)->(AD);
\draw(AC)->(AE);
\draw(AD)->(AE);
\draw(AF)->(AE);
\draw(AA)->(AF);
\draw(AB)->(AF);
\node(AG) at (32.2,2){$2 \cdot 1^2, (123)$};
\node(AH) at (32.2,0){$(143)$};
\draw (AG)->(AH);
\node(AI) at (35.5,2){$1 \cdot 2 \cdot  1, (123)$};
\node(AJ) at (35.5,0){$(142)$};
\draw (AI)->(AJ);
\node(AIa) at (39.5,2){$1^2 \cdot 2, (123)$};
\node(AJa) at (39.5,0){$(1423)$};
\draw (AIa)->(AJa);
\node(AK) at (42.5,2){$2 \cdot 1^2, (132)$};
\node(AL) at (42.5,0){$(1324)$};
\draw (AK)->(AL);
\node(AM) at (46,2){$1 \cdot 2 \cdot  1, (132)$};
\node(AN) at (46,0){$(134)$};
\draw (AM)->(AN);
\node(AO) at (49.5,2){$1^2 \cdot 2, (132)$};
\node(AP) at (49.5,0){$(124)$};
\draw (AO)->(AP);

\node(AQ) at (24,0){$(14)$};
\node(AR) at (34,0){$(1234)$};
\node(AS) at (37.5,0){$(1432)$};
\node(AT) at (29.5,0){$(13)(24)$};
\node(AU) at (21.8,0){$(1342)$};
\node(AV) at (16.3,0){$(1243)$};


\draw (H)--(K);
\draw(G) -- (J);
\draw (C)--(O);
\draw(C) -- (A);
\draw (C) -- (D);
\draw (B) -- (W);
\draw(S) -- (M);
\draw (R) -- (L);
\draw (P) -- (V);
\draw (AQ) -- (AF);
\draw (AT) -- (AIa);
\draw (AT) -- (AK);
\draw (AR) -- (AG);
\draw (AR) -- (AI);
\draw (AS) -- (AM);
\draw (AS) -- (AO);
\draw (AJa) -- (AO);
\draw (AL) -- (AC);
\draw (A) -- (U);
\draw (O) -- (U);
\draw (AIa) -- (Z);
\draw (Z) -- (AK);

\end{tikzpicture}
\end{tiny}
\end{figure}
\end{landscape}

\section{Matrix Semi-invariants}

In this section we study a different but related group action. Let $H = \SL_p \times \SL_p$ act on $\M_p$ by the formula
\[(h_1,h_2) \cdot A = h_1 A h_2^{-1}\] and extend the action diagonally to $\vv:=\M_p^n$. We assume $p>1$. The algebra $\C[\vv]^H$ is called the {\it algebra of $p \times p$ matrix semi-invariants.} These are the semi-invariants of the $n$-Kronecker quiver which we described in the introduction with dimension vector $(p,p)$. 

When $n=1$, for any $p$, $\C[\vv]^H$ has a single generator $\det(X_1)$, and when $n=2$, a generating set is given by the coefficients of the polynomial $\det(t_1X_1+t_2X_2) \in \C[\vv][t_1,t_2]$. In both cases, the ring of invariants is polynomial, see, for instance \cite{Happel2, Happel1}.

 For $n \geq 3$ and $p \geq 2$, the algebra of invariants $\C[\vv]^H$ is not a polynomial ring. While an infinite vector space basis is known for all $p$ and $n$, a minimal generating set for $\C[\vv]^H$ for all $n \geq 3$ is known only in the case $p=2$. The generating set, due to Domokos \cite{DomokosPoincare}, has cardinality $n$+$\binom{n}{2}+\binom{n}{4}$. Domokos also showed that this set is a minimal separating set \cite{DomokosSemi}.

The purpose of this section is to prove the following theorem:

\begin{Theorem}\label{semiconj} Let $n \geq 3$ and $p \geq 4$. Then the algebra of matrix semi-invariants contains no polynomial or hypersurface separating set.
\end{Theorem}

Recalling that the $n$-Kronecker quiver is finite for $n=1$ and tame for $n=2$, this establishes Conjecture \ref{conj} for this quiver.

We will need to know the dimension of $\C[\vv]^H$. Applying Kac's formula \cite[Proposition~4]{KacQuivers} once more we find:
\begin{prop}\label{dimsem}
For all $p \geq 2$ and $n \geq 3$, we have $\dim(\C[\vv]^H) = (n-2)p^2+2$. 
\end{prop}

Now the proof of Theorem \ref{semiconj} is a straightforward application of a result of Domokos: 
for $n \geq 2$ consider the morphism $\rho: \M_p^n \rightarrow \M_p^{n+1}$ given by
\[\rho (A_1,A_2, \ldots, A_{n-1}) =  (A_1,A_2, \ldots, A_{n-1},I)\] where $I$ is the $p \times p$ identity matrix. By \cite[Proposition~4.1]{DomokosRelative}, the induced morphism
\[\rho^*: \C[\M_p^{n}]^H \rightarrow \C[\M_p^{n-1}]^G\] is surjective. This can be used to show that for any separating set $S \subseteq \C[\M_p^{n}]^H$, $\rho^*(S) \subseteq \C[\M_p^{n-1}]^G$ is a separating set, see \cite[Corollary~6.3]{DomokosSemi}. 

\begin{proof}[Proof of Theorem \ref{semiconj}]
Let $n \geq 3$, $p \geq 4$ and suppose $S \subseteq \C[\M_p^n]^H$ is a separating set. Then $\rho^*(S)$ is a separating set for $\C[\M_p^{n-1}]^G$. Since $n-1 \geq 2$, by Theorem \ref{mainsepvar} we have
\[|S| \geq (n-2)p^2+p.\]
Since $\dim(\C[\M_p^n])^H = (n-2)p^2+2$ and $p-2 \geq 2$, we get the desired result. 
\end{proof}

\begin{rem}
If we were interested in finding the smallest $p$ such that $\C[\M_p^n]^H$ contains no polynomial or hypersurface separating set for given $n$ we could take $p=2$ when $n \geq 5$ and $p=4$ for $n=3$ and $n=4$. One might expect that $p=3$ would suffice for $n=4$, but this cannot be proved using the methods above.
\end{rem}

\bibliographystyle{plain}
\bibliography{MyBib}

\end{document}